\documentclass[12pt,english,reqno]{article}
\usepackage{amsfonts}
\usepackage{amsmath, amsthm, amssymb,color}
\usepackage{latexsym}
\usepackage{txfonts}
\usepackage{bm}
\usepackage{graphicx}
\usepackage{graphics}
\usepackage{epsfig}
\usepackage[english]{babel}
\usepackage[T1]{fontenc}
\usepackage{mathrsfs}
\usepackage[a4paper]{geometry}
\numberwithin{equation}{section}
\allowdisplaybreaks
\usepackage{setspace}
\usepackage{enumerate}
\newtheorem{lemma}{Lemma}[section]

\newtheorem{theorem}[lemma]{Theorem}
\newtheorem{proposition}[lemma]{Proposition}
\newtheorem{definition}[lemma]{Definition}
\newtheorem{corollary}[lemma]{Corollary}
\newtheorem{example}[lemma]{Example}
\newtheorem{exercise}[lemma]{Exercise}
\newtheorem{remark}[lemma]{Remark}
\newtheorem{fig}[lemma]{Figure}
\newtheorem{tab}[lemma]{Table}
\newtheorem{assumption}[lemma]{Assumption}

\newcommand{\bth}{\begin{theorem}}
\newcommand{\ethe}{\end{theorem}}
\newcommand{\bre}{\begin{remark}\em }
\newcommand{\ere}{\end{remark}}
\newcommand{\ble}{\begin{lemma}}
\newcommand{\ele}{\end{lemma}}
\newcommand{\bde}{\begin{definition}}
\newcommand{\ede}{\end{definition}}
\newcommand{\bco}{\begin{corollary}}
\newcommand{\eco}{\end{corollary}}
\newcommand{\bpr}{\begin{proposition}}
\newcommand{\epr}{\end{proposition}}

\newcommand{\bexer}{\begin{exercise}}
\newcommand{\eexer}{\end{exercise}}

\newcommand{\bexam}{\begin{example}\rm  }
\newcommand{\eexam}{ \end{example}}

\newcommand{\bfi}{\begin{fig}}
\newcommand{\efi}{\end{fig}}

\newcommand{\btab}{\begin{tab}}
\newcommand{\etab}{\end{tab}}

\def\E{{\mathbb{E}}}
\def\P{{\mathbb{P}}}
\def\R{{\mathbb{R}}}

\def\N{{\mathbb{N}}}

\def\Z{{\mathbb{Z}}}

\def\S{{\mathbb{S}}}

\def\B_e{B_{\eta}(e)}

\renewcommand{\a}{\alpha }

\newcommand{\8}{\infty}

\newcommand{\ov}{\overline}

\newcommand{\wt}{\widetilde}

\definecolor{darkblue}{rgb}{0,0,1}
\definecolor{darkgreen}{rgb}{0,1,0}
\definecolor{darkred}{rgb}{1, 0,0}

\newcommand{\bfu}{{\bf u}}

\newcommand{\bfM}{{\bf M}}

\newcommand{\beao}{\begin{eqnarray*}}
\newcommand{\eeao}{\end{eqnarray*}\noindent}

\newcommand{\beam}{\begin{eqnarray}}
\newcommand{\eeam}{\end{eqnarray}\noindent}

\newcommand{\beqq}{\begin{equation}}
\newcommand{\eeqq}{\end{equation}\noindent}

\newcommand{\bce}{\begin{center}}
\newcommand{\ece}{\end{center}}

\newcommand{\barr}{\begin{array}}
\newcommand{\earr}{\end{array}}

\newcommand{\vague}{\stackrel{\lower0.2ex\hbox{$\scriptscriptstyle
                    \it{v} $}}{\rightarrow}}
\newcommand{\weak}{\stackrel{\lower0.2ex\hbox{$\scriptscriptstyle
                    \it{w} $}}{\rightarrow}}
\newcommand{\what}{\stackrel{\lower0.2ex\hbox{$\scriptscriptstyle
                    \it{\hat{w}} $}}{\rightarrow}}

\newcommand{\bdis}{\begin{displaymath}}
\newcommand{\edis}{\end{displaymath}\noindent}

\newcommand{\wh}{\widehat}

\newcommand{\bfx}{{\bf x}}
\newcommand{\bfX}{{\bf X}}
\newcommand{\bfB}{{\bf B}}

\newcommand{\bfA}{{\bf A}}

\newcommand{\bfa}{{\bf a}}

\newcommand{\bali}{\begin{align}}
\newcommand{\eali}{\end{align}}

\def\hsymb#1{\mbox{\strut\rlap{\smash{\Huge$#1$}}\quad}}

\begin{document}

\bibliographystyle{ecta}
\title{Characterization of the tail behavior of a class of BEKK processes: A stochastic recurrence equation approach}
\date{\today}

\author{Muneya Matsui\thanks{\noindent Department of Business Administration, Nanzan University, 18 Yamazato-cho Showa-ku Nagoya, 466-8673, Japan. {\tt mmuneya@nanzan-u.ac.jp}.  $^\dagger$Department of Economics, University of Copenhagen. Øster Farimagsgade 5, DK-1353 Copenhagen K, Denmark. {\tt rsp@econ.ku.dk} \newline  Matsui's research is partly supported by the JSPS Grant-in-Aid
 for Young Scientists B (16k16023). Pedersen is grateful for support from the Carlsberg Foundation.} \ and Rasmus S\o ndergaard Pedersen$^\dagger$}

\maketitle
\begin{abstract}
We provide new, mild conditions for strict stationarity and ergodicity of a class of BEKK processes. By exploiting that the processes can be represented as  multivariate stochastic recurrence equations, we characterize the tail behavior of the associated stationary laws. Specifically, we show that the each component of the BEKK processes is regularly varying with some tail index. In general, the tail index differs along the components, which contrasts most of the existing literature on the tail behavior of multivariate GARCH processes. 
\vspace{2mm} 
 \end{abstract}
\newpage

\section{Introduction}
In this paper we present novel results about the tail properties for the stationary solution to a class of multivariate conditionally heteroskedastic BEKK processes. Specifically, with $X_t\in \R^d$ we consider BEKK-ARCH (BEKK$(q,0,l)$) processes of the form
\begin{align}
 X_t &= H_t^{1/2}Z_t,\quad t\in \N,  \label{eq:BEKK1}\\
 H_t &= C+ \sum_{i=1}^{q}\sum_{j=1}^l A_{ij} X_{t-i}X_{t-i}' A_{ij}', \label{eq:BEKK2}
\end{align}
where $(Z_t:t\in \N)$ is i.i.d., $Z_t\sim N(0,I_d)$, with $I_d$ the $d \times d$ identity matrix, $C$ is a $d\times d$
positive definite matrix, $A_{ij}\in M(d,\R)$ (the set of $d\times d$ real matrices) for $i=1,...,q$ and $j=1,...,l$, and $X_0,...,X_{-(q-1)}\in \R^d$ are some initial values. This class of processes was originally introduced by Engle and Kroner (1995)\nocite{engle:kroner:1995}. By relying on results for stochastic recurrence equations (SREs), we find a new, mild condition for the existence of an almost surely unique stationary solution to the process in (\ref{eq:BEKK1})-(\ref{eq:BEKK2}). In the case where $l=q=1$ this stationarity condition is given explicitly in terms of the spectral radius of the matrix $A_{11}$, similar to the stationarity condition found by Nelson (1990)\nocite{nelson:1990} for one-dimensional ARCH processes. Next, again relying on results for SREs, we demonstrate that for various specifications of the matrices $A_{ij}$ and various values of $q$ and $l$ that each component (of the stationary solution) to (\ref{eq:BEKK1})-(\ref{eq:BEKK2}) is regularly varying with some index of regular variation, or tail index, $\alpha_i>0$, $i=1,..,d$. Importantly, we show that the tail indexes may in general be different, which contrasts most of the existing body literature on regularly varying solutions to multivariate GARCH processes, where the tail indexes are assumed to be the same along the components of $X$, see e.g. St{\u{a}}ric{\u{a}} (1999)\nocite{starica:1999} and Pedersen (2016)\nocite{pedersen:2016}. Cases of component-wise different tail indexes in the context of multivariate GARCH-type processes are considered in recent articles by Matsui and Mikosch (2016), for constant conditional correlation (CCC) GARCH processes, and Pedersen and Wintenberger (2018) for the process in (\ref{eq:BEKK1})-(\ref{eq:BEKK2}) with $q=l=1$ and $A_{11}$ diagonal (i.e. Diagonal BEKK-ARCH processes). The results in the present paper extend the theory in Pedersen and Wintenberger (2018) in several directions: for $q=1$ and $l\ge1$ we consider the component-wise tail behavior of $X_t$ for cases where the matrices $A_{11},\ldots,A_{1l}$ are simultaneous diagonalizable or simultaneous triangularizable. These cases include several interesting special cases such as triangular $A_{1j}$ and cases where $X_t$ stacks univariate ARCH(1) processes. For $q\ge 1$ we rely on recent results by Guivarc'h and Le Page (2016)\nocite{guivarch:lepage:2016} in order to characterize the tail behavior of $X_t$.  

In a vast amount of applications within quantitative economics and finance, it is well-documented that certain time series exhibit power law tails, see e.g. Loretan and Phillips (1994)\nocite{loretan:phillips:1994} and Gabaix (2009)\nocite{gabaix:2009}. Classic examples of such time series are the series of daily returns on publicly traded shares of stocks; Cont (2001)\nocite{cont:2001} and Ibragimov et al. (2015)\nocite{ibragimov:ibragimov:walden:2015}. In addition to exhibiting extreme values, such return series do typically exhibit conditional heteroskedasticity. The latter has led to an entire research area on univariate and multivariate GARCH models, and it is by now well-known that certain GARCH random variables are heavy tailed, see e.g. Davis and Mikosch (2009) for a discussion on regular variation on univariate GARCH variables and Pedersen and Wintenberger (2018) for references on heavy tailed multivariate GARCH variables. 

In addition to providing new results about the properties of a class of BEKK-ARCH processes in (\ref{eq:BEKK1})-(\ref{eq:BEKK2}), we conjecture that our results are important for obtaining a better understanding of the properties of the quasi-maximum likelihood (QML) estimators for the BEKK class of models. In particular, Avarucci et al$.$ (2013)\nocite{avarucci:beutner:zaffaroni2013} have shown that for a particular class of BEKK-ARCH models (with $q=l=1$), as considered in the present paper, the log-likelihood score contribution has a finite variance if and only if the second-order moments of $X_t$ are finite. Hence, standard arguments used to prove asymptotic normality of QML estimators rely on the assumption that $X_t$ has finite variances. Such condition may not necessarily be satisfied in practice. For instance, Ibragimov et al. (2015, Section 3.2) document that daily returns on certain emerging market foreign exchange rates may have tail index less than two, and hence infinite variance. Likewise, as argued in Pedersen and Rahbek (2014)\nocite{pedersen:rahbek:2014}, the much applied two-step covariance targeting estimator, that relies on computing the sample unconditional covariance matrix of $X_t$, does only seem to obey a Gaussian limiting distribution (at the usual $\sqrt{T}$-rate) provided that at least the fourth-order moments of $X_t$ are finite. In order to derive the limiting distributions of the aforementioned estimators in the case where the moment restrictions on $X_t$ are not satisfied, it appears essential to have results for the tail behavior of $X_t$, as done by Pedersen (2016)\nocite{pedersen:2016} who consider stable limit theory for the variance targeting estimator for multivariate constant conditional correlation (CCC) GARCH models.

The remainder of the paper is organized as follows. In Section \ref{sec:SRE} we state that BEKK-ARCH process can be represented as a stochastic recurrence equation, and we provide a new, mild condition for strict stationarity. In Section \ref{sec:results} we provide a brief overview of recent results on the tail behavior of BEKK-ARCH processes, and we outline our main contributions. Section \ref{sec:preliminaries} contains results on regularly varying random variables and one-dimensional SREs. In Sections \ref{sec:diagonalization} and \ref{sec:triangularization} we present results on tail behavior of BEKK-ARCH processes of order $q=1$ for the cases where the collection of matrices $\{A_{11},\ldots,A_{1l}\}$ is simultaneously diagonalizable and simultaneous triangularizable, respectively. In Section \ref{sec:ARCHQ} we present theory for BEKK-ARCH processes of arbitrary order $q\ge 1$. We provide concluding remarks in Section \ref{sec:conclusion}.

We end this section by providing some definitions and notation used throughout the paper. We let $M(n,\R)$ denote the space of $n \times n$ real matrices. For any column vector $x\in \R^n$ let $|x|$ denote any vector norm of $x$. For any real matrix $A$, let $\Vert A \Vert$ denote the operator norm $\Vert A \Vert = \sup_{x:|x|=1}|Ax|$. We let $\S^{n-1}$ denote the unit sphere in $\R^n$, i.e. $\S^{n-1} = \{x\in \R^n : |x|=1 \}$. For $x\in\R$, $x_+ = \max\{x,0\}$ and $x_- = \max\{-x,0\}$.

\section{The BEKK process as a stochastic recurrence equation}\label{sec:SRE}
In this section we state the stochastic recurrence equation (SRE) representation of the BEKK process in (\ref{eq:BEKK1})-(\ref{eq:BEKK2}). We use the SRE to state a mild condition for the existence of a stationary solution to the process. To the best of our knowledge, this result is new in terms of BEKK processes.

Consider the process in (\ref{eq:BEKK1})-(\ref{eq:BEKK2}). For
$i=1,...,q$ and $j=1,...,l$, let $(m_{i,j,t}:t\in \Z)$ be an
i.i.d. process with  $m_{i,j,t}$ univariate standard normal,
$m_{i,j,t}\sim N(0,1)$, and let $(m_{i,j,t}:t\in \Z)$ and
$(m_{r,s,t}:t\in \Z)$ be mutually independent for all $i\ne r$ and $j\ne
s$. Let $(B_t:t\in\Z) $ be an i.i.d. process with
$B_t\sim N(0,C)$ 
 and mutually independent of $(m_{i,j,t}:t\in \Z)$ for all $i,j$. With $Y_t = (X_{t}',...,X_{t-(q-1)}')'$, noting that $Z_t$ is Gaussian, it holds that
  \begin{eqnarray}
    V_t &=& M_t V_{t-1}+Q_t, \label{eq:SRE_V}
  \end{eqnarray}
  where
  \begin{eqnarray}
    M_t &=& \left(
               \begin{array}{cccc}
                 M_{1,t} & M_{2,t} & \dots & M_{q,t} \\
                 I_d &  &  & 0 \\
                  & \ddots &  & \vdots \\
                  &  & I_d & 0 \\
               \end{array}
             \right), \label{eq:def_M}
  \end{eqnarray}
$Q_t = (B_t',0',\ldots'0')'$, $M_{i,t}=\sum_{j=1}^{l} m_{i,j,t}A_{ij}$
for $i=1,...,q$. In order to show that there exists a stationary solution to the BEKK process, we make the following assumption.
\begin{assumption} \label{ass:lyapunov}
  With $M_t$ defined in (\ref{eq:def_M}), let $\gamma$ denote the top Lyapunov exponent associated with the process in (\ref{eq:SRE_V}), i.e.
  \begin{eqnarray*}
    \gamma &=& \inf_{n\in\N}n^{-1}\E [\log \Vert M_1\cdots M_n\Vert ].
  \end{eqnarray*}
  It holds that $\gamma < 0$.
\end{assumption}

Under Assumption \ref{ass:lyapunov}, and noting that $\E [(\log\Vert M_t\Vert)_+]<\infty$ and 
$\E [({\log | Q_t|})_+]<\infty$, we obtain the following result by an application of Theorem 4.1.4 of Buraczewski et al. (2016) (BDM henceforth)\nocite{buraczewski:damek:mikosch:2016}:
\begin{theorem}
  Suppose that Assumption \ref{ass:lyapunov} is satisfied. Then there exists an almost surely unique strictly stationary ergodic causal solution to the stochastic recurrence equation in (\ref{eq:SRE_V}). In particular, there exists a strictly stationary ergodic solution to the BEKK process in (\ref{eq:BEKK1})-(\ref{eq:BEKK2}).
\end{theorem}

\begin{remark}
\label{rem:statinarity}
  Note that for the case $d=q=l=1$ the BEKK process in
 (\ref{eq:BEKK1})-(\ref{eq:BEKK2}) is a univariate ARCH(1) process,
 i.e. $X_t = (C+A_{11}^2X_{t-1}^2)^{1/2}Z_t$ where $C>0$ and $A_{11}^2\ge 0$ are scalars, and $Z_t\sim i.i.d.N(0,1)$. Nelson (1990)\nocite{nelson:1990} showed that a necessary and sufficient condition for the existence of a
 stationary solution to such process is that $E[\log(A_{11}^2Z_t^2)]<0$,
 i.e. that $A_{11}^2 < \exp(-\psi(1)+{\log 2})=3.56...$, where $\psi$
 denotes the digamma function. As recently noticed by
 Pedersen and Wintenberger (2018)\nocite{pedersen:wintenberger:2018}, one can show that for the case $q=l=1$ a sufficient condition for the existence of a stationary solution to (\ref{eq:BEKK1})-(\ref{eq:BEKK2}) is that $\rho(A_{11}\otimes A_{11})< 3.56...$. In this case, the process in (\ref{eq:BEKK1})-(\ref{eq:BEKK2}) (which is a Markov chain for $q=1$) is geometrically ergodic.
\end{remark}
\begin{remark}
As noted by Nicholls and Quinn (1982, Corollary
 2.1.1)\nocite{nicholls:quinn:1982}, a sufficient condition for stationarity, stronger than Assumption \ref{ass:lyapunov}, is that  $\rho(\E[M_t \otimes M_t])<1$, where $\rho$ denotes the spectral radius. See also Francq and Zako\"ian (2010, Section 11.3)\nocite{francq:zakoian:2010} and Boussama et al. (2011)\nocite{boussama:fuchs:stelzer:2011} for sufficient conditions for stationarity of BEKK-GARCH processes.  
\end{remark}
\begin{remark}
The BEKK process in \eqref{eq:BEKK1}-\eqref{eq:BEKK2} could be extended
 by an autoregressive term such that $X_t = \Phi X_{t-1} + H_t^{1/2}Z_t$
 with $H_t$ given by \eqref{eq:BEKK2}. Such process, which one may
 denote a vector double autoregressive (DAR) process, has been studied
 by Nielsen and Rahbek (2014)\nocite{nielsen:rahbek:2014}, see also Ling
 and Li (2008)\nocite{ling:li:2008} and the references therein for
 details on one-dimensional DAR processes. The vector DAR process has an
 SRE representation of the form \eqref{eq:SRE_V}-\eqref{eq:def_M} with
 $M_{1,t}= \Phi +\sum_{j=1}^{l} m_{1,j,t}A_{1j}$.   Note that the
 process may have a strictly stationary solution even if the matrix
 $\Phi - I$ has reduced rank, in contrast to standard vector
 autoregressive processes of order one. In the remainder of this paper
 we focus on the BEKK processes of the form
 \eqref{eq:BEKK1}-\eqref{eq:BEKK2}, i.e. with $\Phi = 0$, but emphasize
 that the results in the following sections are straightforward to adapt
 to certain vector DAR processes. Hence, we note that certain vector DAR
 processes are indeed heavy-tailed, as conjectured by Nielsen and Rahbek
 (2014, Remarks 5 and 6). 
\end{remark}

Having shown that there exists a strictly stationary solution to the class of BEKK processes in (\ref{eq:BEKK1})-(\ref{eq:BEKK2}), we turn to characterizing the tail-properties of the associated stationary law of the processes. We start out by providing an overview of existing results as well as our new results.

\section{Existing results and our contributions}\label{sec:results}
Our objective is to consider the (component-wise) tail-behavior of $X_t$ given by various BEKK-ARCH processes of the form \eqref{eq:BEKK1}-\eqref{eq:BEKK2}. Recently, Pedersen and Wintenberger (2018)\nocite{pedersen:wintenberger:2018} considered the tail-behavior of $X_t$ for $q=1$ under the following conditions for $M_t$ defined in (\ref{eq:def_M}):

\begin{enumerate}[(a)]
  \item $M_t$ is invertible (almost surely) and has a positive Lebesgue density on $M(d,\R)$. 
  \item $M_t$ is a similarity (almost surely). Specifically, they consider the case where $l=1$ and $A_{11} = aO$ with $a$ a positive constant and $O$ an orthogonal matrix. This includes the well-known scalar BEKK process, by setting $O=I_d$.
  \item $l=1$ and $A_{11}$ is diagonal such that $M_t$ is diagonal. This is the well-known Diagonal BEKK process.
\end{enumerate}
For the first two types of processes, by relying on results due to
Alsmeyer and Mentemeier (2012) \nocite{alsmeyer:mentmeier2012} and
Buraczewski et al. (2009)\nocite{buraczewski:damek:guivarch2009},
respectively, they show that (under suitable conditions) $X_t$ is
multivariate regularly varying with each component having the same tail
index; we refer the reader to the monograph by Resnick
(2007)\nocite{resnick:2007} for more details on multivariate regular
variation. For the Diagonal BEKK process the tail indexes of the
components of $X_t$ differ whenever the diagonal elements of $A_{11}$
differ in modulus. In order to understand this property, we note that
for the diagonal case with $l=1$, with $\tilde A_{ii}$ denoting the $i$th diagonal element of $A_{11}$,
\begin{align} \label{eq:stacksre1}
  \left(
    \begin{array}{c}
      X_{1,t} \\
      \vdots \\
      X_{d,t} \\
    \end{array}
  \right)
  = & \left(
        \begin{array}{c}
          \tilde A_{11}m_{1,1,t}X_{1,t-1} + Q_{1,t} \\
          \ldots \\
          \tilde A_{dd}m_{1,1,t}X_{d,t-1}+ Q_{d,t} \\
        \end{array}
      \right).
\end{align}
Hence each component of $X_t$ obeys a one-dimensional SRE, and the component-wise tail indexes may be determined by Kesten-Goldie theory, see  Lemma \ref{ref:lemma:basic1} in the next section.

We consider the tail-behavior for larger classes of BEKK processes. In particular, we study in detail the following cases:
\begin{enumerate}[(1)]
  \item $q=1$, $l\ge 1$ and the matrices $A_{11},...,A_{1l}$ are simultaneously diagonalizable. This includes the important special case where $l=1$ and $A_{11}$ is full or triangular and diagonalizable. Another special case is when $X_t$ stacks $d$ (potentially independent) one-dimensional ARCH(1) processes.
  \item $q=1$, $l\ge1$ and the matrices $A_{11},...,A_{1l}$ are simultaneously triangularizable. This includes the special case where $l=2$ and $A_{11}$ and $A_{12}$ are triangular but not simultaneously diagonalizable.
  \item $q\ge 1$ and the distribution of the matrix $M_t$ satisfies certain irreducibility and contraction conditions in the spirit of Guivarc'h and Le Page (2016)\nocite{guivarch:lepage:2016}.
\end{enumerate}

 For case (1), considered in Section \ref{sec:diagonalization}, the strategy is to consider a suitable transformation of $X_t$. To fix ideas, in the case $l=1$, $PA_{11}P^{-1} =:D$ is diagonal, such that $Y_t = Dm_{1,t}Y_{t-1}+PQ_t$ with $Y_t=PX_t$. Here $Y_t$ is of the form \eqref{eq:stacksre1}, such that each $Y_{i,t}$ forms an SRE. We then use Lemma \ref{ref:lemma:basic1} to characterize the component-wise tail behavior of $Y_t$. This enables us to study the tail-behavior of $X_t=P^{-1}Y_t$, by carefully applying results for sums of regularly varying random variables (see Lemma \ref{lem:RVconvolution} in the next section).

 For case (2), considered in Section \ref{sec:triangularization}, suppose that $d=2$ and
 \begin{align} \label{eq:bivariatesre1}
   \left(
      \begin{array}{c}
        X_{1,t} \\
        X_{2,t} \\
      \end{array}
    \right)
    &= \left(
        \begin{array}{cc}
          M_{11,t} & M_{12,t} \\
          0 & M_{22,t} \\
        \end{array}
      \right)
    \left(
      \begin{array}{c}
        X_{1,t-1} \\
        X_{2,t-1} \\
      \end{array}
    \right)
    +
    \left(
      \begin{array}{c}
        Q_{1,t} \\
        Q_{2,t} \\
      \end{array}
    \right),
 \end{align}
with $M_{11,t},M_{12,t},M_{22,t}$ non-degenerate. In this case, we see that $X_{2,t}$ obeys an SRE and its tail behavior is obtained via Lemma \ref{ref:lemma:basic1} below. However, $X_{1,t}$ does not obey an SRE. In this case, depending on the properties of $M_{11,t}$, $X_{1,t}$ may inherit the tail shape of $X_{2,t}$ or it has fatter tails than $X_{2,t}$. The characterization of the tail behavior of $X_{1,t}$ is non-trivial and requires new technical arguments, extending the recent results by Damek et al$.$ (2019)\nocite{damek:matsui:swiatkowski:2017} who study the component-wise tail behavior of $\R_+^2$-valued SREs of the form \eqref{eq:bivariatesre1}.

For case (3), studied in Section \ref{sec:ARCHQ}, we show that, under suitable conditions, the BEKK-ARCH($q$) process satisfies some irreducibility and contraction conditions, recently considered by Guivarc'h and Le Page (2016)\nocite{guivarch:lepage:2016}. In particular, one may note that for $q>1$ the distribution of $M_t$ is singular with respect to the Lebesgue measure on $M(dq,\R)$, and hence the approach used in Pedersen and Wintenberger (2018) (see case (a) above) cannot be applied. 

In the next section, we provide a brief overview of results for one-dimensional regularly varying distributions and SREs.

\section{Preliminaries}\label{sec:preliminaries}
The following definitions and results can be found in the recent monograph of BDM \nocite{buraczewski:damek:mikosch:2016}. The results are essential for obtaining the results in the following sections.
For functions $f,g:\R\to\R$, $f(x)\sim g(x)$ means that $\lim_{x\to\infty}f(x)/g(x)\to 1$. A positive measurable function $f$ on $(0,\infty)$ is said to be regularly varying with index $\kappa\in\R$, if for any constant $c>0$, $\lim_{x\to\infty}f(cx)/f(x)=c^{\kappa}$. We say that an $\R$-valued random variable $X$ is regularly varying with index $\alpha \ge 0$ if the function $f(x) = \P(|X|>x)$ is regularly varying with index $-\alpha$ and there exist constants $p,q\ge0$ such that $p+q=1$ and
\begin{eqnarray}\label{eq:def:constants:RV}
  \lim_{x\to\infty}\frac{\P(X>x)}{\P(|X|>x)} = p \quad \text{and} \quad  \lim_{x\to\infty} \frac{\P(X\le -x)}{\P(|X|>x)} = q.
\end{eqnarray}
Note that if $X$ is regularly varying with index $\alpha >0$, then $\E[|X|^{\delta}]<\infty$ for any $0\le\delta<\alpha$, and $\E[|X|^{\tilde \delta}]=\infty$ for any $\tilde \delta > \alpha$. 

The following result is a generalization of Breiman's (1965)\nocite{breiman:1965} lemma, and is useful for characterizing the product of a regularly varying random variable and a lighter-tailed random variable.
\begin{lemma}
\label{lem:Breiman}
Let $X$ and $Y$ be independent random variables.
Assume that $X$ is regularly varying with index $\alpha>0$, and that there exists an $\varepsilon>0$ such that
 $\E|Y|^{\alpha+\varepsilon}<\infty$. Then $XY$ is regularly varying with index $\alpha$. In particular,
\begin{eqnarray}
 \lim_{x\to\infty}\frac{\P(XY>x)}{\P(|X|>x)} = p \E Y_+^\alpha +q \E
 Y_-^\alpha, \quad \text{and} \quad \lim_{x\to\infty}\frac{\P(XY<-x)}{\P(|X|>x)} = p \E Y_-^\alpha +q \E
 Y_+^\alpha, \label{breiman_eq}
\end{eqnarray}
where the constants $p$ and $q$ are given by \eqref{eq:def:constants:RV}.
\end{lemma}

\begin{proof}
 Note that $X = X_+ - X_-$ and $Y = Y_+ - Y_-$. Hence for $x>0$ we have
 \begin{align}
  \P(XY>x) &=\P((X_+-X_-)(Y_+-Y_-)>x) = \P(X_+Y_+>x)+\P(X_-Y_->x), \label{partition}
  \end{align}
and the first part of (\ref{breiman_eq}) follows by an application of Breiman's lemma (c.f. Lemma B.5.1 in BDM \nocite{buraczewski:damek:mikosch:2016}) to each term in (\ref{partition}). The second part of (\ref{breiman_eq}) follows by a similar argument.
\end{proof}

The following result states that regular variation is closed under convolution. A proof is given in Section B.6 of BDM\nocite{buraczewski:damek:mikosch:2016}.
\begin{lemma}\label{lem:RVconvolution}
 Let $X$ and $Y$ be random variables, and assume that $X_+$ is regularly varying with index $\alpha > 0$ such that $\P(|Y|>x)=o(\P(|X|>x))$ as $x\to\infty$. Then
 \begin{align}
 	\P(X+Y > x)/\P(X > x) \to 1 \quad \text{as $x\to\infty$.}
 \end{align}
\end{lemma}

Lastly, we state the following lemma about the strictly stationary solution to one-dimensional SREs. The first result on strict stationarity is given in Theorem 2.1.3 of BDM\nocite{buraczewski:damek:mikosch:2016}, but has been stated elsewhere in the literature under similar assumptions, see e.g. Bougerol and Picard (1992)\nocite{bougerol:picard:1992}. The second part on regular variation is given in Theorem 2.4.7 of BDM \nocite{buraczewski:damek:mikosch:2016} and was originally proved by Goldie (1991)\nocite{goldie:1991}.
\begin{lemma}
\label{ref:lemma:basic1}
Let $X_t$ be a $\R$-valued random variable satisfying the SRE
\begin{align}
	X_t = A_t X_{t-1} + B_t,\quad t\in\Z, \label{eq:SRE_onedim}
\end{align}
with $((A_t,B_t):t\in\Z)$ an $\R^2$-valued i.i.d$.$ sequence.

Suppose that $\P(A_t=0)=0$, $-\infty \le {\E[\log
 |A|]<0}$, and ${\E[(\log |B|)_+]<\infty}$. Then there exists an almost surely unique causal ergodic strictly stationary solution to the SRE in (\ref{eq:SRE_onedim}). Let $\P_0$ denote the distribution of the strictly stationary solution.

Suppose in addition that (1) $\P(A_t<0)>0$  and the conditional
 distribution of ${\log |A_t| }$ given ${A_t\ne 0}$ is non-arithmetic, (2)
 there exists an $\alpha > 0$ such that $\E[|A_t|^\alpha]=1$,
 $E[|B_t|^\alpha]<\infty$, and $\E[|A_t|^\alpha({\log |A_t|})_+]<\infty$, and (3) $\P(A_tx+B_t=x)<1$ for all $x\in\R$. 

Let $(A,B)$ have the same distribution as $(A_t,B_t)$. Then the stochastic fixed point equation 
\begin{align}
 X \overset{d}{=} AX + B
\end{align}
has a solution $X$ which is independent of $(A,B)$ and that has distribution $\P_0$. Moreover, there exists a constant $c_+>0$ such that
\begin{align}
 \P_0(X>x)\sim c_+x^{-\alpha} \quad \text{and} \quad \P_0(X<-x)\sim c_+x^{-\alpha},\quad \text{as $x\to \infty$},
\end{align}
where 
\begin{align}
c_+ = \frac{1}{2\alpha m_\alpha} \E[|AX+B|^\alpha-|AX|^\alpha] \quad
 \text{and} \quad m_\alpha = \E[|A|^\alpha{\log |A|}]>0.
\end{align}

\end{lemma}

\section{Simultaneous diagonalization}\label{sec:diagonalization}
 We now consider the BEKK process in (\ref{eq:BEKK1})-(\ref{eq:BEKK2}) for $q=1$, which we denote the BEKK-ARCH(1) process. Specifically, for $t\in\Z$,
\begin{align}
 X_t &= H_t^{1/2}Z_t,\quad Z_t \sim i.i.d.N(0,I_d), \quad H_t = C+ \sum_{i=1}^l A_{i} X_{t-1}X_{t-1}' A_{i}', \label{eq:BEKKARCH1}
\end{align}
  and we note that the process has the SRE representation,
\begin{eqnarray}
    X_t &=& M_t X_{t-1}+Q_t, \quad M_t = \sum_{i=1}^{l} m_{i,t}A_{i},    \label{eq:SRE:BEKKARCH1}
  \end{eqnarray}
where $(m_{i,t}:t\in \Z)$ is an i.i.d$.$ process with $m_{i,t} \sim N(0,1)$, and  $(m_{i,t}:t\in \Z)$ and
$(m_{j,t}:t\in \Z)$ are mutually independent for all $i\ne j$. Moreover, $(Q_t:t\in\Z) $ is an i.i.d$.$ process with $Q_t\sim N(0,C)$
 and mutually independent of $(m_{i,t}:t\in \Z)$ for all $i$.

We consider BEKK-ARCH(1) processes satisfying that the collection
 $\{A_i:i=1,...,l\}$ is simultaneously diagonalizable, i.e$.$ the
 collection satisfies that there exists a real non-singular matrix $P$
 such that $D_i=PA_iP^{-1}$ is diagonal for $i=1,...,l$.\footnote{From
 Theorem 1.3.21 of Horn and Johnson (2013)\nocite{horn:johnson:2013} we
 have that
 {a set of diagonalizable matrices $\{A_i:i=1,...,l\}$, $l\ge 2$}, is simultaneously diagonalizable if and only if any pair of the set commutes.} Note that if $l=1$, we simply have that the matrix $A_1$ should be diagonalizable. We recall here that a sufficient, but indeed not necessary, condition for $A_1$ being diagonalizable is that all its eigenvalues are distinct. Noting that since the collection $\{A_i:i=1,...,l\}$ is simultaneously diagonalizable, we may, using \eqref{eq:SRE:BEKKARCH1}, define $Y_t=PX_t$ such that
\begin{align}\label{eq:SRE:Y:simdiag1}
  Y_t & =  \wt M_t Y_{t-1} + \wt Q_t, \quad \wt M_t =\sum_{j=1}^l m_{j,t}D_j, \quad \wt Q_t = P Q_t.
\end{align}
With $D_{ii,j}$ the $i$th diagonal element of $D_j$, we have that the $i$th component of $Y_t$, $Y_{i,t}$, can be written as an SRE,
\begin{align}\label{eq:SRE:Y:simdiag2}
  Y_{i,t} & = \left(\sum_{j=1}^l m_{j,t}D_{ii,j}\right)Y_{i,t-1} + \wt Q_{i,t},\quad i=1,...,d.
\end{align}
The idea is then to apply Lemma \ref{ref:lemma:basic1} to each component $Y_{i,t}$. Specifically, under certain conditions stated in Theorem \ref{thm:simdiag} below, there exist constants $c_{i,+}>0$ and $\alpha_i^{(Y)}>0$ such that $\P(Y_{i,t}>x)\sim c_{i,+}x^{-\alpha_i^{(Y)}}$ and $\P(Y_{i,t}<-x)\sim c_{i,+}x^{-\alpha_i^{(Y)}}$, i.e. $Y_{i,t}$ is regularly varying with tail index $\alpha_i^{(Y)}$. We have that $X_t=P^{-1}Y_t$, such that with $P^{ij}$ denoting element $(i,j)$ of $P^{-1}$, $X_{i,t}=\sum_{j=1}^{d}P^{ij}Y_{j,t}$. The tail index of $X_{i,t}$ is then obtained by careful investigation of the sum of the regularly varying variables $Y_{j,t}$.

We make the following assumptions that imply strict stationarity of the BEKK-ARCH(1) process and regular variation of $Y_{i,t}$.
\begin{assumption}\label{ass:simdiag}
	Let $(X_t:t\in \Z)$ be the BEKK-ARCH(1) process given in \eqref{eq:BEKKARCH1}. The collection $\{A_i:i=1,...,l\}$ is simultaneously diagonalizable, such that there exist a non-singular $P \in M(d,\R)$ and diagonal matrices $D_1,\ldots,D_l \in M(d,\R)$ such that $D_j = PA_jP^{-1}$ for $j=1,\ldots,l$. Let $D_{ii,j}$ denote the $i$th diagonal element of matrix $D_j$. For $i=1\ldots,d$ there exists $\alpha_i^{(Y)}>0$ such that $\E[|\sum_{j=1}^{l}D_{ii,j}m_{j,t}|^{\alpha_i^{(Y)}}]=1$.   
\end{assumption}
\begin{remark}
Noting that $m_{j,t}$ and $m_{i,t}$ are independent for $i\ne j$, we have that $\E[|\sum_{j=1}^{l}D_{ii,j}m_{j,t}|^{\alpha_i^{(Y)}}]=\E[|({\sum_{j=1}^{l}D_{ii,j}^2})^{1/2}z|^{\alpha_i^{(Y)}}]$ with $z$ a standard normal random variable. Hence it is straightforward to check if $\alpha_i^{(Y)}>0$ in Assumption \ref{ass:simdiag} exists.  	
\end{remark}
We obtain the following theorem. 

\begin{theorem} \label{thm:simdiag}
  Let $(X_t:t\in \Z)$ be the BEKK-ARCH(1) process given in \eqref{eq:BEKKARCH1}. Suppose that Assumption \ref{ass:simdiag} holds. Then the process has an almost surely unique strictly stationary and ergodic solution, $X_t = (X_{1,t},\ldots,X_{d,t})'$. Let $P^{ij}$ denote element $(i,j)$ of $P^{-1}$ and define the
 collection $\mathcal{A}_i = \{\alpha = \alpha_j^{(Y)} : j=1,...,d \
 \text{and} \ P^{ij} \ne 0  \}$. Suppose that $\alpha_i =
 \min\mathcal{A}_i$ has multiplicity one. Then $X_{i,t}$ is regularly varying with index $\alpha_i$.
\end{theorem}

\begin{proof}
We start out by showing that $(X_t:t\in \Z)$ has an almost surely unique
 strictly stationary and ergodic solution. Since $X_t = P^{-1}Y_t$, with
 $Y_t$ given by \eqref{eq:SRE:Y:simdiag1}, it suffices to show that
 $(Y_t:t\in \Z)$ has an almost surely unique strictly stationary and
 ergodic solution. By Theorem 4.4.1 of BDM, this is the case if the top
 Lyapunov exponent $\gamma^{Y} := \inf_{n\in\N}n^{-1}\E [\log \Vert \wt
 M_1\cdots \wt M_n \Vert ] <0$. Let $\gamma_i :=
 \E[\log|\sum_{j=1}^{l}D_{ii,j}m_{j,t}|] $. By Theorem 1.1 of
 Gerencs\'er et al. (2008)\nocite{GGO}, $\gamma^{Y} =
 \max_{i=1,\ldots,d}\gamma_i$. By Jensen's inequality, we have that
 $\gamma_i<0$ for all $i =1,\ldots,d$, and we conclude that $(Y_t:t\in
 \Z)$ has an almost surely unique strictly stationary and ergodic solution.

Next, it is straightforward to show that each $Y_{i,t}$ given by \eqref{eq:SRE:Y:simdiag2} satisfies Lemma \ref{ref:lemma:basic1} under Assumption \ref{ass:simdiag}. Specifically, one may conclude that $Y_{i,t}$ is regularly varying with index $\alpha_i^{(Y)}>0$.  

It remains to characterize the tail behavior of $X_{i,t}=\sum_{j=1}^{d}P^{ij}Y_{j,t}$ for $i=1,...,d$. 
Let $\mathcal{K}_i = \{j=1,...,d:P^{ij}\ne 0\}$ and define the
 collection of component-wise tail indexes of $Y_t$ that are relevant
 for $X_{i,t}$, $\mathcal{A}_i = \{\alpha = \alpha_j^{(Y)} : j\in
 \mathcal{K}_i \}$. When $\alpha_i := \min\mathcal{A}_i$ has
 multiplicity one, we may without loss of generality assume that 
 $\alpha_i=\alpha_1^{(Y)}$, i.e.  
 $\{1\} = \{j=1,...,d: P^{ij}\ne 0,\alpha_j^{(Y)}=\alpha_i
 \}$.
Then $X_{i,t}=P^{i1}Y_{1,t}+\sum_{j\in \mathcal{K}_i\setminus\{1\}}P^{ij}Y_{j,t}$. Using that each component of $Y_t$ has a symmetric distribution, and by repeated use of Lemma \ref{lem:RVconvolution}, we conclude that $P^{i1}Y_{1,t}$ has a lower tail index than $\sum_{j\in \mathcal{K}_i\setminus\{1\}}P^{ij}Y_{j,t}$ such that $\P(X_{i,t}>x)\sim c_ix^{-\alpha_i}$ and $\P(X_{i,t}<-x)\sim c_ix^{-\alpha_i}$ for some constant $c_i>0$.
\end{proof}

We next consider some applications of Theorem \ref{thm:simdiag}, where it is (implicitly) assumed that Assumption \ref{ass:simdiag} holds.
\begin{example}
Let $e_i$ denote a $d$-dimensional column vector satisfying, with
 $e_{j,i}$ denoting its $j$th entry, $e_{i,i}=1$ and $e_{j,i}=0$ for
 $j\ne i$. Consider the BEKK process where $q=1,$ $l=d$ and
 $A_i:=A_{1,i}=e_ie_i^\prime a_i$ for some non-zero constant $a_i$,
 $i=1,...,d$. Here $X_t$ stacks $d$
 univariate ARCH(1) processes, potentially correlated. (If the matrix $C$ is diagonal, then the processes are mutually independent.) We note that the matrix $A_i$ is diagonal, and hence that the collection $\{A_i:i=1,...,d\}$  is simultaneously diagonalizable, choosing $P=I_d$. The tail index, $\alpha_i$, of $X_{i,t}$ satisfies $\E[|a_im_{i,t}|^{\alpha_i}]=1$.
\end{example}

\begin{example}
\label{ex32}
  Suppose that $l=1$ and that $A_1$ has non-zero, in modulus distinct real eigenvalues, $D_{11},...,D_{dd}$. Then $A_1$ is diagonalizable, such that for some non-singular $P\in M(d,\R)$, $D_1=:{\rm diag}(D_{11},...,D_{dd}) = PA_1P^{-1}$. Then $\alpha_i^{(Y)}>0$ satisfies $\E[|D_{ii}m_{1,t}|^{\alpha_i^{(Y)}}]=1$. Since, the eigenvalues are distinct in modulus, we have that all $\alpha_i^{(Y)}$ are distinct. We conclude that $X_{i,t}$ has tail index $\alpha_i = \min\{\alpha=\alpha_j^{(Y)} : j=1,...,d \ \text{and} \ P^{ij} \ne 0  \}$. \\
  As a simple example, suppose that $l=1$ and $d=2$ with
  \begin{align*}
    A_1 & = \left(
               \begin{array}{cc}
                 a & c \\
                 0 & b \\
               \end{array}
             \right), \quad \text{$a,b\ne 0$ and $|a|\ne |b|$.}
  \end{align*}
 Then $D_1 = PA_1P^{-1}$, with
 \begin{align*}
   D_1  & =\left(
               \begin{array}{cc}
                 a & 0 \\
                 0 & b \\
               \end{array}
             \right), \quad
   P  = \left(
               \begin{array}{cc}
                 1 & \frac{c}{a-b} \\
                 0 & 1 \\
               \end{array}
             \right),
             \quad \text{and}\quad
   P^{-1} = \left(
               \begin{array}{cc}
                 1 & -\frac{c}{a-b} \\
                 0 & 1 \\
               \end{array}
             \right).
 \end{align*}
  Following Theorem \ref{thm:simdiag}, $\alpha_1^{(Y)}$ and $\alpha_2^{(Y)}$ satisfy respectively $\E[|am_{1,t}|^{\alpha_1^{(Y)}}]= \E[|bm_{1,t}|^{\alpha_2^{(Y)}}]=1$. We have that $X_{2,t}$ has tail index $\alpha_2^{(Y)}$, and $X_{1,t}$ has tail index $\alpha_1^{(Y)} \wedge \alpha_2^{(Y)}$.

\end{example}

\begin{example}
\label{ex33}
	Consider the case $l=d=2$ where,
	\begin{align*}
    A_1 & = \left(
               \begin{array}{cc}
                 a & b \\
                 b & a \\
               \end{array}
             \right) 
             \quad \text{and} \quad 
    A_2 = \left(
               \begin{array}{cc}
                c & 0 \\
                 0 & c \\
               \end{array}
             \right), 
             \quad \text{$a\ne b$}.
  \end{align*}
 We have that $A_1$ and $A_2$ are simultaneous diagonalizable such that $D_1 = PA_1P^{-1}$ and $D_2 = PA_2P^{-1}$  with 
 \begin{align*}
   D_1  & =\left(
               \begin{array}{cc}
                 a-b & 0 \\
                 0 & a+b \\
               \end{array}
             \right), \quad
   D_2   =\left(
               \begin{array}{cc}
                 c & 0 \\
                 0 & c \\
               \end{array}
             \right), \quad          
   P  = \left(
               \begin{array}{cc}
                 -\frac{1}{2} & \frac{1}{2} \\
                 \frac{1}{2} & \frac{1}{2} \\
               \end{array}
             \right),
             \quad \text{and}\quad
   P^{-1} = \left(
               \begin{array}{cc}
                 -1 & 1 \\
                 1 & 1 \\
               \end{array}
             \right).
 \end{align*}
 	With $z$ a standard normal random variable, it holds that $\E[|\sqrt{(a-b)^2+c^2}z|^{\alpha_1^{(Y)}}]= \E[|\sqrt{(a+b)^2+c^2}z|^{\alpha_2^{(Y)}}]=1$ where $\alpha_1^{(Y)}\ne \alpha_2^{(Y)}$, since $a \ne b$. In particular, $\alpha_1^{(Y)}<\alpha_2^{(Y)}$ ($\alpha_1^{(Y)}>\alpha_2^{(Y)}$) if $|a-b|>|a+b|$ ($|a-b|<|a+b|$). We conclude that $X_{1,t}$ and  $X_{2,t}$ have tail index $\alpha_1^{(Y)} \wedge \alpha_2^{(Y)}$.
\end{example}

\begin{example}
In contrast to the previous example, we may for $l=1$ have that $A_1$ has some non-distinct eigenvalues. Suppose that $d=3$, and that with $a,b,c\ne 0$ and $a\ne b$,
    \begin{align*}
    A_1 & = \left(
              \begin{array}{ccc}
                a & 0 & 0 \\
                0 & a & 0 \\
                0 & c & b \\
              \end{array}
            \right).
  \end{align*}
  Then $D_1 = PA_1P^{-1}$ with
  \begin{align*}
    D_1 & = \left(
             \begin{array}{ccc}
               a & 0 & 0 \\
               0 & a & 0 \\
               0 & 0 & b \\
             \end{array}
           \right),
           \quad
           P = \left(
                 \begin{array}{ccc}
                   0 & \frac{c}{a-b} & 0 \\
                   1 & 0 & 0 \\
                   0 &  -\frac{c}{a-b} & 1 \\
                 \end{array}
              \right),\quad \text{and} \quad
               P^{-1} = \left(
                 \begin{array}{ccc}
                  0 & 1 & 0 \\
                   \frac{a-b}{c} & 0 & 0 \\
                   1 & 0 & 1 \\
                 \end{array}
               \right).
             \end{align*}
   In this case, $\alpha_1^{Y}=\alpha_2^{Y}$, satisfying $\E[|am_{1,t}|^{\alpha_1^{Y}}]=1$ and $\alpha_3^{Y}$ satisfies $\E[|bm_{1,t}|^{\alpha_3^{Y}}]=1$. Due to the structure of $P^{-1}$, we have that $X_{1,t}$ and $X_{2,t}$ have tail index $\alpha_1^{Y}$, whereas $X_{3,t}$ has index $\alpha_1^{Y}\wedge \alpha_3^{Y}$, since $a \ne b$.

\end{example}
The above example motivates the following theorem.

\begin{theorem} \label{thm:simdiag2}
 Let $(X_t:t\in\Z)$ be a BEKK-ARCH(1) process given by \eqref{eq:BEKKARCH1} with $l=1$. Suppose that Assumption \ref{ass:simdiag} holds such that the process is strictly stationary. Let $D_1=PA_1P^{-1}$, and let $\mathcal{A}_i$ be defined as in Theorem
 \ref{thm:simdiag}. Moreover, let $\alpha_i = \min\mathcal{A}_i$, and let $\mathcal{G}_i=\{j=1,...,d: P^{ij}\ne 0,\alpha_j^{(Y)}=\alpha_i \}$. With $D_{jj}$ the $j$th diagonal element of $D$, suppose that $D_{jj}=D_{kk}$ for all $k,j\in \mathcal{G}_i$. Then $X_{i,t}$ is regularly varying with index $\alpha_i$.
\end{theorem}
\begin{proof}
  With $Y_t=PX_t$, we have that $Y_{i,t}$ has tail index
 $\alpha_i^{(Y)}>0$, satisfying $\E[|D_{ii}m_{1,t}|^{\alpha_i^{(Y)}}]=1$. Let $\mathcal{G}_i^\dagger = \{j=1,...,d:P^{ij}\ne 0\}\setminus \mathcal{G}_i$. It holds that $X_{i,t} = \sum_{j\in \mathcal{G}_i}P^{ij}Y_{j,t} + \sum_{j\in \mathcal{G}_i^\dagger}P^{ij}Y_{j,t}.$ With $j\in \mathcal{G}_i$ let $\lambda = D_{jj}$. Then $\sum_{j\in \mathcal{G}_i}P^{ij}Y_{j,t} = \lambda m_{1,t} \sum_{j\in \mathcal{G}_i}P^{ij}Y_{j,t-1} + \sum_{j\in \mathcal{G}_i}P^{ij}\wt Q_{j,t}$, where $\wt Q_t = PQ_t$. Hence $\sum_{j\in \mathcal{G}_i}P^{ij}Y_{j,t}$ obeys an SRE, and Lemma \ref{ref:lemma:basic1} implies that $\sum_{j\in \mathcal{G}_i}P^{ij}Y_{j,t}$ has tail index $\alpha_i$. By repeated use of Lemma \ref{lem:RVconvolution} we conclude that the tail index of $X_{i,t}$ is $\alpha_i$
\end{proof}

\begin{example}
For $l=1$ and $d=3$, suppose that for  $a,b\ne 0$ and $a\ne b$,
  \begin{align*}
    A_1 & = \left(
              \begin{array}{ccc}
                a & b & b \\
                b & a & b \\
                b & b & a \\
              \end{array}
            \right).
  \end{align*}
  Then $D_1 = PA_1P^{-1}$ with
  \begin{align*}
    D_1 & = \left(
             \begin{array}{ccc}
               a-b & 0 & 0 \\
               0 & a-b & 0 \\
               0 & 0 & a+2b \\
             \end{array}
           \right),
           \
           P = \left(
                 \begin{array}{ccc}
                   -\frac{1}{3} & \frac{2}{3} & -\frac{1}{3} \\
                   -\frac{1}{3} & -\frac{1}{3} & \frac{2}{3} \\
                   \frac{1}{3} & \frac{1}{3} & \frac{1}{3} \\
                 \end{array}
              \right),\ \text{and} \
               P^{-1} = \left(
                 \begin{array}{ccc}
                   -1 & -1 & 1 \\
                   1 & 0 & 1 \\
                   0 & 1 & 1 \\
                 \end{array}
               \right). 
             \end{align*}
            In this case, $\alpha_1^{(Y)}=\alpha_2^{(Y)}$, satisfying $\E[|(a-b)m_{1,t}|^{\alpha_1^{(Y)}}]=1$ and $\alpha_3^{(Y)}$ satisfies $\E[|(a+2b)m_{1,t}|^{\alpha_3^{(Y)}}]=1$. We note that $\alpha_1^{(Y)}=\alpha_2^{(Y)}=\alpha_3^{(Y)}$ if and only if $a=-b/2$.  In light of Theorem \ref{thm:simdiag2}, we have that each component of $X_t$ has tail index $\alpha_1^{(Y)}$ if $\alpha_1^{(Y)}<\alpha_3^{Y}$ (i.e. if $|a-b|>|a+2b|$) and index $\alpha_3^{(Y)}$ if $\alpha_3^{(Y)}<\alpha_1^{(Y)}$ (i.e. if $|a-b|<|a+2b|$).

\end{example}

\section{Simultaneous triangularization}\label{sec:triangularization}
In this section, we consider the BEKK-ARCH(1) process in \eqref{eq:BEKKARCH1} for $d=2$, and $l\ge 1$, where the matrices $A_1,.,,,A_l$
are simultaneous triangularizable in the sense that there exists a nonsingular $P \in M(2,\R)$ such that $U_i = PA_iP^{-1}$ is upper
triangular for all $i=1,..,l$\footnote{By Horn and Johnson (2013, Theorem 2.4.8.7 and the comments thereafter), we have that a  set of square matrices $\{A_i:i=1,...,l\}$, $l\ge 2$
 is simultaneously triangularizable by a unitary matrix if any pair of the set commutes.}.
A special case is that $A_i = U_i$ such that $P=I_2$. Defining $Y_t=PX_t$, we have the SRE representation of the form \eqref{eq:SRE:Y:simdiag1} where $\wt M_t =\sum_{i=1}^{l}m_{it}U_i$ and $\wt Q_t = P Q_t$.
Recall that the original process $X_t$ is easily recovered by
$X_t=P^{-1}Y_t$. We hence study the special case $A_i = U_i$ and $P=I_2$, so
that we work on \eqref{eq:bivariatesre1} with $X_t$ replaced by $Y_t$,
\begin{align}\label{eq:triangularSRE}
   Y_t & = M_t Y_{t-1} + Q_t,
\end{align}
which we may write as
\begin{align}
\label{bivSRE}
  \left(
    \begin{array}{c}
      Y_{1,t}  \\
      Y_{2,t} \\
    \end{array}
  \right)
   & = \left(
        \begin{array}{cc}
          M_{11,t} & M_{12,t} \\
          0 & M_{22,t} \\
        \end{array}
      \right)
           \left(
    \begin{array}{c}
      Y_{1,t-1}  \\
      Y_{2,t-1} \\
    \end{array}
  \right) + \left(
    \begin{array}{c}
      Q_{1,t}  \\
      Q_{2,t} \\
    \end{array}
  \right).
\end{align}
Note that this SRE has the coordinate-wise representation,
\begin{align}
\label{componentSRE1}
 Y_{1,t} &= M_{11,t} Y_{1,t-1} +D_t \\
\label{componentSRE2}
 Y_{2,t} &= M_{22,t} Y_{2,t-1} + Q_{2,t},
\end{align}
where $D_t= M_{12,t}Y_{2,t-1}+Q_{1,t}$.

For notational convenience, we occasionally omit the subscript $0$ in $M_{ij,0},\,Q_{i,0}$ and just write $M_{ij}$ and $Q_i$. Moreover, we define for $t\in\Z$,
\begin{align*}
 \Pi_{t,s}^{(i)}& =\prod_{j=s}^t M_{ii,j},\,t\ge s,\,i=1,2,\quad
 \mathrm{and}\quad \Pi_{t,s}^{(i)}=1,\,t<s,\quad \mathrm{and}\quad \Pi_t^{(i)}=\Pi_{t,1}^{(i)}.
\end{align*}
Since $M_t$ is triangular, it holds that the stationarity condition in Assumption \ref{ass:lyapunov} can be simplified. Specifically, let
\begin{align}
\label{eq:triangular-gamma}
 {\gamma_i = \inf_{n\in\N} n^{-1} \E [ \log
 |\Pi_{n}^{(i)}| ]
 =\E[\log |M_{ii}| ].}
\end{align}
Then by Theorem 1.1 of Gerencs\'er et al. (2008)\nocite{GGO}, Assumption \ref{ass:lyapunov} holds if and only if
\begin{align}
\label{stationarity-triangular}
 \max_{i=1,2} \gamma_i <0.
\end{align}
If $l=1$, this condition for $A_{1,ii},\,i=1,2$ reduces to those stated in Remark \ref{rem:statinarity} for the case $q=l=1$. In line with Lemma \ref{ref:lemma:basic1} we assume that there exist $\alpha_1>0$ and $\alpha_2>0$ such that  
\begin{align}
\label{cond:tails:triangle}
 \E [|M_{11}|^{\alpha_1}]=1 \qquad \text{and} \qquad \E
 [|M_{22}|^{\alpha_2}]=1.
\end{align}
By an application of Jensen's inequality, it is easily concluded that condition \eqref{cond:tails:triangle} implies the stationarity condition 
\eqref{stationarity-triangular}; see also Proposition 2.1 of Damek et al. (2019)\nocite{damek:matsui:swiatkowski:2017} and the proof of Theorem \ref{thm:simdiag}. \\

We now turn to the tail behavior of each component \eqref{componentSRE1} and \eqref{componentSRE2}.
By Gaussianity it is easy to see that
\begin{align}\label{eq:expectationstriangle}
 \E [|M_{ii}|^{\alpha_i} {(\log |M_{ii}| )_+ }]<\infty \qquad \text{and} \qquad
 \E [|Q_{i,0}|^{\alpha_i}]<\infty, \quad i=1,2.
\end{align}
Then the tail behavior of component $Y_{2,0}$ is immediate from Lemma \ref{ref:lemma:basic1}. On the other hand, the tail behavior of $Y_{1,0}$, given by the SRE \eqref{componentSRE1}, is far from trivial to obtain as the random sequence $(M_{11,t},D_t)$ is stationary but not i.i.d$.$ and, moreover, $Y_{1,t-1}$ and $D_{t}$ are dependent. Indeed, standard Kesten-Goldie-type theory (i.e. Lemma \ref{ref:lemma:basic1}) is not applicable for this type of SRE, as recently pointed out by Damek et al. (2019)\nocite{damek:matsui:swiatkowski:2017}, who consider the case of $\R_{+}^2$-valued SREs. 
The following theorem states the tail-properties of $Y_{1,0}$ and $Y_{2,0}$. 
\begin{theorem}
\label{thm:triangular}
 Let $k_i=\E[|M_{ii}|^{\alpha_i}\log |M_{ii}|],\,i=1,2.$ Consider the
 bivariate SRE \eqref{bivSRE} such that \eqref{cond:tails:triangle} holds. Then there exists a stationary solution to the SRE, $Y_0$, that satisfies
 \begin{align*}
  \P (Y_{1,0}>x) \sim \P(Y_{1,0}<-x) \sim \Bigg \{
\begin{array}{ll}
\ov c_{1} x^{-\alpha_1} & \mathrm{if}\ \alpha_1
   < \alpha_2 \\
\tilde c_{1} x^{-\alpha_2} &\mathrm{if}\
   \alpha_1 > \alpha_2,
\end{array}
 \end{align*}
and
\begin{align}
\label{tail:y2}
 \P(Y_{2,0}>x) \sim \P(Y_{2,0}<-x) \sim c_2 x^{-\alpha_2},\qquad x \to \infty,
\end{align}
where $k_i>0$ and
\begin{align*}
 \ov c_1 &= \frac{1}{2\alpha_1 k_1}\E \big[ |M_{11}Y_{1,0}+D_0|^{\alpha_1}-|M_{11}Y_{1,0}|^{\alpha_1} \big], \\
 \tilde c_1 &= c_2 \lim_{s\to \infty} \E
 \big[ |\sum_{i=1}^s \Pi_{0,2-i}^{(1)} \Pi_{-i,1-s}^{(2)} M_{12,-i}
 |^{\alpha_1}\big], \\
 c_2 &= \frac{1}{2\alpha_2k_2} \E \big[ |M_{22}Y_{2,0} +Q_{2,0}|^{\alpha_2}-|M_{11}Y_{2,0}|^{\alpha_2} \big].
\end{align*}
\end{theorem}
\begin{remark}
The above theorem is easily extended to general bivariate SREs of the form \eqref{bivSRE} with $M_t$ and $Q_t$ non-Gaussian, by assuming that \eqref{eq:expectationstriangle} holds, that $\E[|M_{12}|^{\min\{\alpha_1,\alpha_2\}}]<\infty$, and that the law of $\log|M_{ii}|$ conditional on $|M_{ii}|\ne 0$ is non-arithmetic. 	
\end{remark}

\noindent
 In order to prove the theorem, in light of the above discussion, it suffices to establish the tail behavior for $Y_{1,0}$. We emphasize that this is non-trivial as the SRE \eqref{componentSRE1} does not satisfy standard conditions.  We extend the theory recently developed by Damek et al$.$ (2019)\nocite{damek:matsui:swiatkowski:2017} for non-negative SREs to {$\R^2$-valued SREs.} This extension requires lengthy technical arguments given in the Appendix.


\begin{example}
Consider the process given in Example \ref{ex32}. 
Assuming that
$\E[|am_{1,t}|^{\alpha_1}]= \E[|bm_{1,t}|^{\alpha_2}]=1$, 
a direct application of Theorem \ref{thm:triangular} yields the same conclusion in terms of the tail behavior of $X_0$ as in Example \ref{ex32}.
\end{example} 
\begin{example}
\label{ex64}
Consider the SRE in \eqref{eq:triangularSRE} with $l=2$ such that
\begin{align*}
   A_1  & =\left(
               \begin{array}{cc}
                 a & b \\
                 0 & a \\
               \end{array}
             \right), \quad
   A_2  = \left(
               \begin{array}{cc}
                 c & 0 \\
                 0 & \tilde c \\
               \end{array}
             \right),
             \quad a,b,c,\tilde c\neq 0,\,|c|\neq |\tilde c|,\,a+c,a+\tilde c\neq 0.
 \end{align*}
Noting that $A_1$ is non-diagonalizable, $A_1$ and $A_2$ are not simultaneously diagonalizable, but trivially simultaneously triangularizable with $P=I_2$. Suppose that there exist $\alpha_1>0$ and
 $\alpha_2>0$ such that $\E[|am_{1,t}+cm_{2,t}|^{\alpha_1}]= \E[|am_{1,t}+\tilde c
 m_{2,t}|^{\alpha_2}]=1$. Due to Theorem \ref{thm:triangular},
 $X_{1,0}$ has tail index $\alpha_1
 \wedge \alpha_2$  while $X_{2,0}$ has index $\alpha_2$.
\end{example}
In the next example, we consider the case where $l=2$ and $A_1$ and $A_2$ are non-triangular, but simultaneously triangularizable.

\begin{example}
\label{ex65}
 Let $l=2$ and consider the SRE in \eqref{eq:triangularSRE} where 
\begin{align*}
   A_1  & =\left(
               \begin{array}{cc}
                 a & \frac{b-a}{2} \\
                 \frac{a-b}{2} & b \\
               \end{array}
             \right), \quad
   A_2  = \left(
               \begin{array}{cc}
                 a & c \\
                 a-b+c & b \\
               \end{array}
             \right),
             \quad |a| \neq |b|,\,a,b,c\neq 0,\,c\neq \frac{b-a}{2},-a,b.
\end{align*}
Note that $A_1$ and $A_2$ are not commutable (and hence not simultaneously diagonalizable) since 
\[
 [A_1 A_2]_{12} = ac +\frac{b^2-ab}{2} \neq cb + \frac{ab-a^2}{2} = [A_2 A_1]_{12}
\]
where $[\cdot]_{ij}$ is the $ij$ element of matrix in the
 bracket. However, they are simultaneously triangularizable:
 $U_1=PA_1P^{-1}$ and $U_2=P A_2 P^{-1}$ with 
 \begin{align*}
   U_1  & =\left(
               \begin{array}{cc}
                 \frac{a+b}{2} & b-a \\
                 0 & \frac{a+b}{2} \\
               \end{array}
             \right), \quad
   U_2  = \left(
               \begin{array}{cc}
                 a+c & b-a \\
                 0 & b-c \\
               \end{array}
             \right),\quad 
   P  & =\left(
               \begin{array}{cc}
                 \frac{1}{\sqrt{2}} & \frac{1}{\sqrt{2}} \\
                 - \frac{1}{\sqrt{2}} & \frac{1}{\sqrt{2}} \\
               \end{array}
             \right), \quad
   P^{-1}  = \left(
               \begin{array}{cc}
                 \frac{1}{\sqrt{2}} & - \frac{1}{\sqrt{2}} \\
                 \frac{1}{\sqrt{2}} & \frac{1}{\sqrt{2}} \\
               \end{array}
             \right).
\end{align*}
Let $d:=\sqrt{(a+b)^2 /4 +(a+c)^2}$ and $e:=\sqrt{(a+b)^2/4 +(b-c)^2}$, and suppose that for a standard normal random variable $z$, $\E[|d
 z|^{\beta_1}]=\E[|e z|^{\beta_2}]=1$ with $\beta_1\neq
 \beta_2$. Then according to Theorem
 \ref{thm:triangular}, $\alpha_1^{(Y)}=\beta_1 \wedge \beta_2$ and
 $\alpha_2^{(Y)}=\beta_2$. Since $X=P^{-1}Y $ by Lemma \ref{lem:RVconvolution},
$\alpha_1^{(X)}=\alpha_2^{(X)}=\beta_1$ when $\beta_1
 <\beta_2$. However, for $\beta_2<\beta_1$, so that
 $\alpha_1^{(Y)}=\alpha_2^{(Y)}=\beta_2$, it is not clear how to determine the tail behavior of $X_0$, since $Y_{1,0}$ and $Y_{2,0}$ are dependent.
\end{example}

\begin{remark}
(a) 
 Consider the simple SRE in \eqref{eq:triangularSRE} where $l=1$ and $A_1$ non-diagonalizable,
  \begin{align*}
    A_1 & = \left(
               \begin{array}{cc}
                 a & 1 \\
                 0 & a \\
               \end{array}
             \right), \quad a\neq 0.
  \end{align*}
In this case, \eqref{cond:tails:triangle} implies that $\alpha_1 =
 \alpha_2$, and obtaining the tail properties of $Y_{1,0}$ appears to be
 non-trivial task. A similar case has recently been studied by Damek and
 Zienkiewicz (2018)\nocite{damek:zienkiewicz:2017} who consider a SRE of
 the type \eqref{eq:triangularSRE} with $M_{11}$ and $M_{22}$
 non-negative almost surely. We leave the case $\alpha_1 = \alpha_2$ for
 future research.  \\
 (b) Our results here could be extended to the $d$-dimensional case by
 extending recent results for $\R^d_+$-valued SREs by Matsui and
 \'{S}wi\k{a}tkowski (2018)\nocite{matsui:swiatkowski:2018}. We leave this extension to future work. 
\end{remark}

\section{Tail properties of BEKK-ARCH($q$)} \label{sec:ARCHQ}
In this section we consider the tail properties of the BEKK ARCH process of order $q \ge 1$. Recall that this process has the SRE representation given by \eqref{eq:SRE_V}-\eqref{eq:def_M}, and the main idea is to show that the SRE satisfies certain irreducibility and contraction conditions recently considered by Guivarc'h and Le Page (2016)\nocite{guivarch:lepage:2016}, see also Section 4.4.8 of BDM. 

With $M_t$ defined in \eqref{eq:def_M}, let $\P_M$ denote its distribution. Define 
\begin{eqnarray} \label{eq:def:GM}
	G_M=\{s\in M(d,\R) : s = a_1 \cdots a_n, \ a_i \in \mathrm{supp}\P_M, \ i = 1,\ldots,n, \ n\in \N \},
\end{eqnarray}
where $\mathrm{supp}\P_M$ denotes the support of $\P_M$. We initially make the following high-level assumptions (BDM, p.189):
\begin{assumption}\label{ass:ip}
	(a) With $G_M$ defined in \eqref{eq:def:GM}, there exists no finite union $\mathcal{W} = \bigcup_{i=1}^n W_i$ of proper subspaces $W_i \subsetneq \R^{dq}$ such that for any $v \in G_M$, $v \mathcal{W}=\mathcal{W}$. (b) $G_M$ contains a matrix that has a unique largest eigenvalue in modulus with multiplicity one. 
\end{assumption}
Assumption \ref{ass:ip}(a) is an irreducibility condition, and
Assumption \ref{ass:ip}(b) is a contraction condition stating that $G_M$ contains a proximal matrix. In Lemmas \ref{lem:irreducibility}-\ref{lem:contraction} below we state more primitive sufficient conditions for Assumption \ref{ass:ip}. The following theorem states that the stationary solution to the SRE in \eqref{eq:SRE_V}-\eqref{eq:def_M} is multivariate regularly varying; see e.g.Resnick (2007)\nocite{resnick:2007}.

\begin{theorem}\label{thm:ARCHq}
 For the BEKK ARCH process of order $q \ge 1$ with SRE representation given by \eqref{eq:SRE_V}-\eqref{eq:def_M}, suppose that Assumptions \ref{ass:lyapunov} and \ref{ass:ip} hold, that $\P[\mathrm{det}(M_{q,t})=0]=0$, and that there exists $\alpha >0$ such that $\inf_{n\in \N}(\E[\Vert M_1\dots M_n \Vert^{\alpha}])^{1/n}=1$. Then the stationary solution, $V_t$, to the SRE is multivariate regularly varying with index $\alpha$, i.e. there exists a probability measure $\P_{\Theta}$ on $\S^{dq-1}$ such that 
 \begin{eqnarray}\label{eq:MRV}
 	\frac{\P(\vert V_t \vert > s x, \wt V_t \in \cdot)}{\P(\vert V_t \vert > x)} \overset{w}{\rightarrow} s^{-\alpha}\P_{\Theta}(\cdot), \quad \text{as $x \to \infty$}, \quad s>0, \quad \wt V_t = V_t/\vert V_t \vert,
 \end{eqnarray}
 where $\overset{w}{\rightarrow}$ denotes weak convergence.
\end{theorem}
The multivariate regular variation in \eqref{eq:MRV} implies that for any $y\in \S^{dq-1}$, $P(y'V_t > x) \sim c(y) x^{-\alpha}$ as $x\to \infty$, where $c(y)$ may depend on $y$ and  $c(\tilde y)>0$ for some $\tilde y \in \S^{dq-1}$. Moreover, $|V_t|$ is regularly varying with index $\alpha$. 
\begin{proof}
 The theorem is proved by verifying the conditions of Theorem 5.2 of
 Guivarc'h and Le Page (2016)\nocite{guivarch:lepage:2016}; see also
 Theorem 4.4.18 in BDM. It suffices to show that (i) $M_t$ is invertible
 almost surely, (ii) for all $x\in \R^{dq}$, $\P(M_t x + Q_t = x )<1$,
 and (iii) $\E[\Vert M_t \Vert^{\alpha + \delta} ]<\infty$, $\E[\Vert
 M_t \Vert^{\alpha} \Vert M_t^{-1} \Vert^{\delta} ]<\infty$, and 
{$\E[| Q_t |^{\alpha + \delta} ]<\infty$} for some  $\delta >0$. Condition (i) is clearly satisfied as $\mathrm{det}(M_t) = \mathrm{det}(M_{q,t}) \ne 0 $ almost surely. Condition (ii) is immediate as $M_t$ and $Q_t$ are independent and $Q_t$ is non-degenerate. Condition (iii) holds by noting that the elements of $M_t$ and $Q_t$ are Gaussian and an application of H\"older's inequality, choosing $\delta >0$ sufficiently small.
\end{proof}

{The following lemmas give sufficient conditions for Assumption \ref{ass:ip}(a).} 
\begin{lemma} \label{lem:irreducibility}
	With $M_{1,t},\ldots,M_{q,t}$ the random matrices in \eqref{eq:def_M}, let $M_t^{(1,q)}$ denote the $d \times dq$ matrix given by 
	\begin{eqnarray}\label{eq:def:M1q}
		M_t^{(1,q)} = (M_{1,t},\ldots,M_{q,t}).
	\end{eqnarray}
 	Suppose that for any non-zero $x \in \R^{dq}$ the distribution of $M_t^{(1,q)} x$ has a density with respect to the Lebesgue measure strictly positive on $\R^d$. Then Assumption \ref{ass:ip}(a) holds. 	
\end{lemma}
 \begin{proof}
 	 The proof extends the arguments in Section 4.4.9 of BDM to arbitrary dimension $d$. The strategy is to show that the only space that satisfies $v \mathcal{W}=\mathcal{W}$ for all $v \in G_M$ is $\mathcal{W} = \R^{dq}$. We show this by contradiction by assuming that the space $\mathcal{W}$ is not equal to $\R^{dq}$. Specifically, $\mathcal{W} = \bigcup_{i=1}^n W_i$ for proper subspaces $W_i \subsetneq \R^{dq}$. Let $x$ be some non-zero vector from one of the subspaces, and consider the partition $x = (x_1',\ldots, x_q')'$, $x_i \in \R^d$. Let $M_{(1)}^{(1,q)},\ldots,M_{(q)}^{(1,q)}$ denote $q$ independent copies of $M_{t}^{(1,q)}$, and likewise let $M_{(1)},\ldots, M_{(q)}$ denote $q$ independent copies of $M_t$, where the first $d$ rows of $M_{(i)}$ are given by $M_{(i)}^{(1,q)}$. Then 
 	 \begin{eqnarray*}
 	 	M_{(1)}x = \left(
 	  	             \begin{array}{c}
                 M_{(1)}^{(1,q)}x \\
                 x_1 \\
                 \vdots \\
                  x_{q-1} \\
               \end{array}\right).
 	 \end{eqnarray*}
 	 Since $M_{(1)}^{(1,q)}x$ has a Lebesgue density strictly positive on $\R^d$, necessarily there must be a subspace $W_{i_1}$ satisfying $\mathcal{V}_1 :=\{(z_1',x_1',\ldots,x_{q-1}')'$: $z_1\in \R^d  \}  \subset W_{i_1}$. Next, the action $M_{(2)}$ on $\mathcal{V}_1$ yields, 
 	 \begin{eqnarray*}
 	 	M_{(2)}v_1 = \left(
 	  	             \begin{array}{c}
                 M_{(2)}^{(1,q)}v_1 \\
                 z_1 \\
                 x_1 \\
                 \vdots \\
                  x_{q-2} \\
               \end{array}\right),\quad v_1=(z_1',x_1',\ldots,x_{q-1}')'.
 	 \end{eqnarray*}
 	 Using again that $M_{(2)}^{(1,q)}v_1$ has a Lebesgue density strictly positive on $\R^d$ (for $v_1 \ne 0$), there must exist a subspace $W_{i_2}$ such that $\mathcal{V}_2 :=\{(z_1',z_2',x_1',\ldots,x_{q-2}')'$: $z_1,z_2\in \R^d  \}  \subset W_{i_2}$. By repeating these arguments we conclude that one of the subspaces $W_i$ equals $\R^{dq}$.
 \end{proof}

For $q=1$ with $d=l=2$ the following lemma is useful. The lemma is also
applicable for checking the irreducibility condition in Alsmeyer and
Mentemeier (2012); see Alsmeyer and Mentemeier (2012, Condition (A4)) \nocite{alsmeyer:mentmeier2012} and Theorem 4.4.15 of BDM.

\begin{lemma}
\label{lem:irreducibility:l2}
 Let $M_t\,(=M_{1,t}=M_{t}^{(1,1)})=m_{1,t}A_1+m_{2,t}A_2$
with $A_1,\,A_2 \in M(2,\R)$ where $m_{i,t}\sim N(0,1),\,i=1,2$ are
 independent, namely SRE \eqref{eq:SRE_V} reduces to SRE \eqref{eq:SRE:BEKKARCH1}. 
With $x=(x_1,x_2)' \in
 \R^2$ we write $x^{(n)}=\Pi_{i=1}^n M_i x$. For any $x\ne 0$ assume that there exists $n
 \in \N$ such that almost surely 
\[
 A_1 x^{(n)},\,A_2 x^{(n)} \neq 0,\quad \text{and} \quad A_1 x^{(n)}
 \neq k A_2 x^{(n)}\quad \text{for any}\, k\in \R, 
\]
i.e. the vectors $A_1 x^{(n)}$ and $A_2 x^{(n)}$ are not parallel. Then Assumption
 \ref{ass:ip}(a) holds. 
\end{lemma}

\begin{proof}
 It suffices to observe that since $ m_{i,n+1}\sim N(0,1),\,i=1,2$ are independent, 
\[
 M_{n+1} x^{(n)} = m_{1,n+1}A_1 x^{(n)} + m_{2,n+1} A_2 x^{(n)} 
\] 
may take any value in $\R^2$. Thus Assumption \ref{ass:ip}(a) follows.
\end{proof}

The following lemma gives a sufficient condition for Assumption  \ref{ass:ip}(b).
\begin{lemma} \label{lem:contraction}
 With $M_t$ given in \eqref{eq:def_M}, suppose that for any $i=1,\ldots,q$, $M_{i,t}$ has a density with respect to the Lebesgue measure on $M(d,\R)$ that is strictly positive on a neighborhood around zero. Then Assumption \ref{ass:ip}(b) holds.
 \end{lemma}
 \begin{proof}
 	The result is immediate by noting that $M_{i,t}$ and $M_{j,t}$ are independent for $i\ne j$.
 \end{proof}
The BEKK-ARCH process in the following example satisfies Assumption \ref{ass:ip}.
\begin{example}
	Consider the case $d=q=2$ and $l=4$ where
	  \begin{align*}
    A_{i1} & = \left(
               \begin{array}{cc}
                 a_{i1} & 0 \\
                 0 & 0 \\
               \end{array}
             \right),
              A_{i2}  = \left(
               \begin{array}{cc}
                 0 & 0 \\
                 a_{i2} & 0 \\
               \end{array}
             \right),
             A_{i3}  = \left(
               \begin{array}{cc}
                 0 & a_{i3} \\
                 0 & 0 \\
               \end{array}
             \right),
             A_{i4}  = \left(
               \begin{array}{cc}
                 0 & 0 \\
                 0 & a_{i4} \\
               \end{array}
             \right), \ i=1,2,
  \end{align*}
  for some non-zero $a_{ij}$. Since all elements of the matrices $M_{1,t}$ and $M_{2,t}$ are independent and Gaussian, we have that $M_{1,t}$ and $M_{2,t}$ have  densities strictly positive on $M(d,\R)$. Moreover, for any non-zero $x \in \R^{4}$ the distribution of $M_t^{(1,2)} x$ has a density that is strictly positive on $\R^2$. By Lemmas and \ref{lem:irreducibility} and \ref{lem:contraction}, the process satisfies Assumption \ref{ass:ip}. We may also note that $\P[\mathrm{det}(M_{2,t})=0]=0$. For suitable values of constants $a_{ij}$ Assumption \ref{ass:lyapunov} holds and there exists $\alpha >0$ such that $\inf_{n\in \N}(\E[\Vert M_1\dots M_n \Vert^{\alpha}])^{1/n}=1$. Hence, under these conditions, Theorem \ref{thm:ARCHq} applies. 
\end{example}

In the following example we consider the case where $q=1$ and $d=l=2$
and the matrices $A_1$ and $A_2$ are neither simultaneously
diagonalizable nor simultaneously triangularizable. 

\begin{example}
Suppose that $q=1,\,d=l=2$ and consider \eqref{eq:SRE:BEKKARCH1} with 
 $M_t=m_{1,t}A_1 +m_{2,t}A_2$ where 
\begin{align*}
   A_1  & =\left(
               \begin{array}{cc}
                 a & b \\
                 b & a \\
               \end{array}
             \right), \quad
   A_2  = \left(
               \begin{array}{cc}
                 a & b \\
                 -b & -a \\
               \end{array}
             \right),
             \quad |a|> |b|,\,a,b \neq 0.
\end{align*}
Since the eigenvalues of $A_1$ and $A_2$ are respectively $a\pm b$ and
 $\pm\sqrt{a^2-b^2}$, $A_1$ and $A_2$ are diagonalizable. However, due to
 non-commutability, they are not simultaneously diagonalizable. 
 We check the conditions of Theorem \ref{thm:ARCHq}. Define a set of
 vectors $\mathcal V= \{x=(x_1,x_2)'\in \R^2 \mid x= (\pm b,\mp
 a)',\,(\pm a,\mp b)' \}.$ For non-zero $x \in \R^2 \setminus \mathcal
 V$, 
 $A_1 x$ and $A_2 x$ are linearly independent, while for $x\in
 \mathcal V$, $M_1 x$ is proportional to either $(0,1)'\not \in \mathcal
 V$ or $(1,0)' \not \in \mathcal V$. 
Thus via Lemma \ref{lem:irreducibility:l2}, Assumption \ref{ass:ip}(a)
 holds. For Assumption \ref{ass:ip}(b) take $n=1$ in $G_M$ and observe
 that 
$
 |M_{1,t}-\lambda I|=0\, \Leftrightarrow \, \lambda^2 -2 a m_{1,t}
 \lambda + (a^2-b^2) (m_{1,t}^2 -m_{2,t}^2) =0,
$
so that the eigenvalues may differ. Note that $\mathrm{det}
 (M_t)=(a^2-b^2)(m_{1,t}^2+m_{2,t}^2) \ne 0$ almost surely. In order to assure the existence of
 $\alpha>0$ such that $\inf_{n\in \N}(\E\|M_1\cdots
 M_n\|^\alpha)^{1/n}=1$, as in Remark 4.4.16 of  BDM, it is enough to
 assume that for some $p>0$, $\E[(\lambda_{\min}(M_t M_t'))^{p/2}]\ge 1$
 where $\lambda_{\min}(M_t M_t')$ is the smallest eigenvalue of
 $M_t M_t'$. However, in view of the characteristic equation of $M_t
 M_t'$:
\[
 \lambda^2 -2(a^2+b^2)(m_{1,t}^2+m_{2,t}^2)\lambda +(a^2-b^2)(m_{1,t}^2
 -m_{2,t}^2)=0, 
\] 
non-negative eigenvalues of $M_tM_t'$ are proportional to $(a,b)$ and we
 can choose appropriate values. In a similar manner, we can adjust
 $(a,b)$ so that Assumption \ref{ass:lyapunov} holds. 
\end{example}




\section{Concluding remarks}\label{sec:conclusion}
We conclude by stating some important directions for future research. For the cases considered in Sections \ref{sec:diagonalization} and \ref{sec:triangularization}, we focused on the \emph{component-wise} tail behavior of $X_t$. Ideally, one would also be interested in obtaining results for the dependence structure of $X_t$, as this can be used for establishing stable limit theory for $X_t$, see e.g. Section 4.5 of BDM and Pedersen and Wintenberger (2018, Section 4).  As the components, or marginals, of $X_t$ have different indexes of regular variation, it seems appealing to find conditions such that $X_t$ is \emph{non-standard regularly varying} in the sense of Resnick (2007, Section 6.5.6) or \emph{vector scaling regularly varying} as introduced in Pedersen and Wintenberger (2018). Finding such conditions for general multivariate SREs is a tremendous task and an active area of research.

The SRE representation for the BEKK-ARCH process in \eqref{eq:SRE_V}-\eqref{eq:def_M} relies crucially on the assumption that the noise variable $Z_t$ is Gaussian and that the process is of the ARCH-type, i.e. $H_t$ does not include lagged values of itself. Characterizing the tail behavior of general GARCH-type BEKK processes with non-Gaussian noise is indeed an interesting open issue that inherently seems to require another approach than relying on as SRE representation of the processes.


\section*{Appendix: Proof of Theorem \ref{thm:triangular}}

Throughout the proof, we apply the following component-wise series representations of the unique stationary solution to \eqref{bivSRE}, $Y_t =(Y_{1,t},Y_{2,t})$, which are given by
\begin{align}
\label{componentwiseSol1}
 Y_{1,t} &= \sum_{i=1}^\infty \Pi_{t,t+2-i}^{(1)} D_{t+1-i},\quad
 \text{where}\quad  D_t= M_{12,t}Y_{2,t-1}+Q_{1,t}, \\\label{componentwiseSol2}
 Y_{2,t} &= \sum_{i=1}^\infty \Pi_{t,t+2-i}^{(2)} Q_{2,t+1-i}.
\end{align}
We start out by verifying that these representations are well-defined. The expression \eqref{componentwiseSol2} follows easily from \eqref{componentSRE2}, and the series converges absolutely almost surely (see e.g. proof of Theorem
2.1.3 in BDM \nocite{buraczewski:damek:mikosch:2016}). Turning to
\eqref{componentwiseSol1}, consider the SRE \eqref{componentSRE1}. Note
that the random element $Y_{2,t}$ is measurable w.r.t. the
$\sigma$-field generated by $(M_{t-s},Q_{t-s})_{s\in \Z_{-}}$ (see
Section 2.6 of Straumann, 2005\nocite{S}), and so is
$(M_{11,t},Q_{1,t})$. Thus $(M_{11,t},D_t)$ with
$D_t=M_{11,t}Y_{2,t-1}+Q_t$ is also measurable, where we notice that
component-wise measurability is equivalent to the measurability of a
vector. Then due to e.g. Proposition 4.3 of Krengel (2011)
\nocite{krengel:2011}, $(M_{11,t},D_t)_{t\in \Z}$ is a stationary and
ergodic sequence. Using that $E[\log|M_{11}|]<0$ and {$E[(\log |Y_{2,0}|)_+]<\infty$}, it follows by Theorem 1 of Brandt (1986) \nocite{brandt:1986} that \eqref{componentwiseSol1} is the unique stationary solution to \eqref{componentSRE1} and that the series converges absolutely almost surely. Since \eqref{bivSRE} has a unique solution, we conclude that the solution, $(Y_{1,t},Y_{2,t})$ 
of \eqref{componentwiseSol1} and \eqref{componentwiseSol2} is the component-wise series representation.  



Next, we consider a decomposition of $Y_{1,0}$ in terms of the solutions to two other SREs. In particular, consider the SREs given by 
\begin{align}
\label{componentwiseSRE1}
\wh Y_{1,t} =  M_{11,t} \wh Y_{1,t-1}+Q_{1,t}, 
\end{align}
\begin{align}
\label{componentwiseSRE2}
\wt Y_{1,t} =  M_{11,t} \wt Y_{1,t-1}+\wt D_{t}, \quad \wt D_{t}=M_{12,t}Y_{2,t-1}. 
\end{align}
By the same reasoning as above, these SREs have unique solutions, respectively,
\begin{align}
\wh Y_{1,0} =  \sum_{i=1}^\infty \Pi_{0,2-i}^{(1)}Q_{1,1-i}, 
\end{align}
and
\begin{align}
\wt Y_{1,0} =  \sum_{i=1}^\infty \Pi_{0,2-i}^{(1)} M_{12,1-i} Y_{2,-i}, 
\end{align}
where the series converge absolutely almost surely. 
Thus we have that
\begin{align}
\label{decomp1}
 Y_{1,0}= \wt Y_{1,0} + \wh Y_{1,0}.
\end{align}
\begin{proof}
 Throughout $c$ denotes a generic positive constant. \\
 {\bf (i) Case $\alpha_1>\alpha_2$.} Our strategy is that we further
 decompose $\wt Y_{1,0}$
 into several parts. By comparing their tail behaviors we specify the
 dominant term, which determines the tail behavior of $Y_1$. 
 First we show the general scheme. The detailed tail asymptotics
 of the dominant and negligible terms are given later. 
Without loss of generality, we consider the upper tail $\P(Y_1>x)$. Observe that in \eqref{decomp1}, $\wh Y_1$ is regularly varying with
 index $\alpha_1$, i.e. 
\[
 \P(\wh Y_{1,0} >x) \sim {\wh c_1} x^{-\alpha_1}
\]
and we turn to the tail properties of $\wt Y_1$. We
decompose $\wt Y_1$ into three parts,
\begin{equation}
 \wt Y_{1,0} = \Big( \underbrace{\sum_{i=1}^s}_{\wt Z_s} +\underbrace{\sum_{i=s+1}^\infty}_{\wt Y^s} \Big)\, \Pi_{0,2-i}^{(1)}
 M_{12,1-i}Y_{2,-i} =: \underbrace{\wt Y_{s,1} + \wt Y_{s,2}}_{\wt Z_s} + \wt
 Y^s, \label{decomp:wtX3}
\end{equation}
where in $\wt Z_s$ we apply the iteration of the SRE for $Y_{2,-i}$ until
 time $-s<-i$,
\[
 Y_{2,-i} = \Pi_{-i,1-s}^{(2)} Y_{2,-s}+ {\sum_{k=0}^{s-i-1}
 \Pi^{(2)}_{-1,1-i-k} Q_{2,-i-k}},
\]
and substitute this into $\wt Z_{s}$, so that
\begin{align}
\label{decomp:wtXunders}
 \wt Z_s
 &= \underbrace{\sum_{i=1}^s \Pi_{0,2-i}^{(1)}M_{12,1-i}
 \Pi_{-i,1-s}^{(2)}
 Y_{2,-s}}_{\wt Y_{s,1}} +
 \underbrace{\sum_{i=1}^s \Pi_{0,2-i}^{(1)} M_{12,1-i}
 \sum_{k=0}^{s-i-1}\Pi_{-i,1-i-k}^{(2)}Q_{2,-i-k}}_{\wt Y_{s,2}}.
\end{align}
The idea is then to study the tail behavior of each term in \eqref{decomp:wtX3}. Specifically, we later show that there are constants $C>0,\,0<q<1$ such that for every $s$
\begin{equation}
\label{geometric}
\P (|\wt Y^s|>x)\leq C q^sx^{-\a_2}.
\end{equation}
Moreover, for a fixed (but arbitrary) $s$
\begin{equation}\label{second}
\lim _{x\to \8}\P (|\wt Y_{s,2}|>x)x^{\a _2}=0
\end{equation}
\noindent and
\begin{equation}\label{first}
\lim _{x\to \8}\P (\wt Y_{s,1}>x)x^{\a _2}=c_2w_s,
\end{equation}
where $c_2$ is that in \eqref{tail:y2} and
\begin{align*}
 w_s 
&= \E \Big|
\sum_{i=1}^s \Pi_{0,2-i}^{(1)} \Pi_{-i,1-s}^{(2)} M_{12,-i}
\Big|^{\a_2}
\end{align*}
with
\begin{align}
\label{bounded}
 \sup_{s\in \N} w_s<\infty.
\end{align}
Hence we note that the term $\wt Y_{s,1}$ is the dominating term in \eqref{decomp:wtX3}. Now, using \eqref{decomp:wtX3}, we have that
\begin{align*}
 \P(Y_1>x) &\le \P(\wt Y_{s,1}>(1-3\varepsilon)x)+
\P(\wh Y_1>\varepsilon x) +
\P(\wt
 Y_{s,2}>\varepsilon x) +\P(\wt Y^s >\varepsilon x), \\
 \P(Y_1>x) &\ge \P(\wt Y_{s,1}>(1+3\varepsilon)x)-
\P(\wh Y_1<-\varepsilon x) -
\P(\wt
 Y_{s,2}<-\varepsilon x) -\P(\wt Y^s <-\varepsilon x). 
\end{align*}
Then after multiplying $x^{\alpha_2}$ to both sides of inequalities, we
 make the limit operation of $x\to\infty$ and obtain 
\begin{align}
\label{ineq:main}
 (1+3\varepsilon)^{-\alpha_2} c_2 w_s - C q^s &\le \liminf_{x\to\infty}
 x^{\alpha_2} \P(Y_1>x) \\
 &\le \limsup_{x\to\infty} x^{\alpha_2}\P(Y_1>x) \nonumber \\
 &\le    (1-3\varepsilon)^{-\alpha_2} c_2 w_s + C q^s. \nonumber
\end{align}
In the upper and lower bounds, 
we take a converging subsequence
 $w_{s_k}$ of $w_s$ and then $\varepsilon \downarrow 0$. 
Due to \eqref{ineq:main} the limit for $k\to\infty$ satisfies 
\[
 c_2 \lim_{k\to \infty} w_{s_k} = \lim_{x\to \infty}
 x^{\alpha_2}\P(Y_1>x). 
\]
Since every converging subsequence converges to the same limit, we have
\[
 \lim_{x\to \infty} x^{\alpha_2} \P(Y_1>x) =c_2 \lim_{s\to\infty} w_s.
\]
It remains to prove \eqref{geometric}--\eqref{bounded}.
We begin with \eqref{bounded} and recall that $\E|M_{11}|^{\alpha_2}<1$. If $\alpha_2\le 1$ then
\[
 w_s \le \sum_{i=1}^\infty (\E|M_{11}|^{\alpha_2})^{(i-1)}
 \E|M_{12}|^{\alpha_2}<\infty
\]
and if $\alpha_2>1$ then
\[
 w_s^{1/\alpha_2} \le \sum_{i=1}^\infty (\E
 |M_{11}|^{\alpha_2})^{(i-1)/\alpha_2}
 (\E|M_{12}|^{\alpha_2})^{1/\alpha_2}<\infty.
\]
Since the bounds above do not depend on $s$, \eqref{bounded}
 follows.
Concerning \eqref{geometric} we use Markov inequality and conditioning in the
 following way
\begin{align*}
 x^{\alpha_2} \P(|\wt Y^s|>x) &= x^{\alpha_2} \P \big(
|\sum_{i=s+1}^\infty \Pi_{0,2-i}^{(1)} M_{12,1-i}Y_{2,-i}|>x
\big) \\
&\le x^{\alpha_2} \P\big(
\sum_{i=1}^\infty | \Pi_{0,2-(s+i)}^{(1)} M_{12,1-(s+i)}Y_{2,-(s+i)}| >x
\big) \\
& \le \sum_{i=1}^\infty x^{\alpha_2} \P\big(
|\Pi_{0,2-(s+i)}^{(1)} M_{12,1-(s+i)}Y_{2,-(s+i)}|> x\, i^{-\mu}\,/ \zeta(\mu)
\big) \\
&= \sum_{i=1}^\infty \E \big[ x^{\alpha_2}
\P\big(
G_i | Y_{2,-(s+i)} | >x \mid G_i
\big)
\big],
\end{align*}
where $G_i = \zeta(\mu)\,i^\mu\, |\Pi_{0,2-(i+s)}^{(1)}M_{12,1-(s+i)}|$ with $\mu>1$ and
 $\zeta(\cdot)$ is zeta function. The integrability above follows from
 the fact $\sum_{i=1}^\infty i^{-\mu}=\zeta(\mu)$. Notice that
 $G_i$ and $|Y_{2,-(i+s)}|$ are independent, $\E G_i^{\alpha}<\infty$
 and there is a constant $c$ such that for every $x>0$,
\[
 \P(|Y_{2,-(i+s)}|>x) \le c x^{-\alpha_2}.
\]
Hence it follows from 
\[
 \E\big[
x^{\alpha_2} \P\big(
G_i |Y_{2,-(s+i)}| >x \mid G_i
\big)
\big] \le c \E G_i^{\alpha_2}
\]
that 
\begin{align*}
 x^{\alpha_2} \P(|\wt Y^s|>x) &\le c \sum_{i=1}^\infty \E
 G_i^{\alpha_2} \\
 & = c \sum_{i=1}^\infty \E |\Pi_{0,2-(s+i)}|^{\alpha_2} \E
 |M_{12,1-(s+i)}|^{\alpha_2} \zeta(\mu)^{\alpha_2}\,i^{\alpha_2 \mu} \\
 & \le \underbrace{c \E|M_{12}|^{\alpha_2} \zeta(\mu) ^{\alpha_2}
 \sum_{i=1}^\infty (\E|M_{11}|^{\alpha_2})^{i+1} i^{\alpha_2\mu}}_{C}
 \,\cdot (\E|M_{11}|^{\alpha_2})^s, 
\end{align*}
where $\sum_{i=1}^\infty (\E|M_{11}|^{\alpha_2})^i i^{\alpha_2\mu}
 <\infty$ since $\E|M_{11}|^{\alpha_2}<1$. Putting $q=\E
 |M_{11}|^{\alpha_2}$, we obtain \eqref{geometric}.
We prove \eqref{second} by showing that $\E|\wt
 Y_{s,2}|^{\alpha_2}<\infty$ for any fixed $s$. 
 We work on the expression in \eqref{decomp:wtXunders}. 
Recall that $(M_t,Q_t)$ are i.i.d. so that $\Pi_{0,2-i}^{(1)}\, M_{12,1-i}$
 and $\sum_{k=0}^{s-i-1}\Pi_{-i,1-i-k}^{(2)}
 Q_{2,-i-k},\,i=1,2,\ldots,s$ are independent. We further recall that 
 $\E|M_{22}|^{\alpha_2}=1,\,\E|M_{11}|^{\alpha_2}<1,\,\E|M_{12}|^{\alpha_2}<\infty$
 and $\E|Q_2|^{\alpha_i}<\infty$. For $\alpha_2>1$, by Minkowski's inequality,
 \begin{align*}
  \E |\wt Y_{s,2}|^{\alpha_2} &= \E \big|
 \sum_{i=1}^s \Pi_{0,2-i}^{(1)} M_{12,-i} \sum_{k=0}^{s-i-1}
  \Pi_{-i,1-i-k}^{(2)} Q_{2,-i-k}
\big|^{\alpha_2} \\
 &\le \Big[
 \sum_{i=1}^s \Big\{
 (\E |M_{11}|^{\alpha_2})^{i+1} \E |M_{12}|^{\alpha_2} \E
  |Q_2|^{\alpha_2} (s-i)^{\alpha_2}
\Big\}^{1/\alpha_2}
\Big]^{\alpha_2}<\infty
 \end{align*}
and for $\alpha_2 \le 1$, by sub-additivity,
\begin{align*}
 {\E |\wt Y_{s,2}|^{\alpha_2}} &\le \sum_{i=1}^s \E \big|
 \Pi_{0,2-i}^{(1)} M_{12,-i} \sum_{k=0}^{s-i-1}\Pi_{-i,1-i-k}^{(2)} Q_{s,-i-k}
\big|^{\alpha_2} \\
& \le \sum_{i=1}^s (\E |M_{11}|^{\alpha_2})^{i+1} \E |M_{12}|^{\alpha_2}
 \sum_{k=0}^{s-i-1} (\E |M_{22}|^{\alpha_2})^k \E|Q_2|^{\alpha_2} \\
&= \E |M_{12}|^{\alpha_2} \E |Q_2|^{\alpha_2} \sum_{i=1}^s (\E
 |M_{11}|^{\alpha_2})^{i+1}(s-i)<\infty.
\end{align*}
Hence for fixed $s$ we have \eqref{second}. Finally we observe
\[
 \wt Y_{s,1} = R_s Y_{2,-s},
\]
where $R_s:= \sum_{i=1}^s \Pi_{0,2-i}^{(1)}M_{12,-i} \Pi_{-i,1-s}^{(2)}$
 and $Y_{2,-s}$ are independent. Hence Breiman's lemma (Lemma \ref{lem:Breiman}) yields
\[
 \lim_{x\to\infty}x^{\alpha_2}\P(\wt Y_{s,1}>x)= c_2 w_s,
\]
which is \eqref{first}. This finishes the first part of the proof. \\

\noindent
{\bf (ii) Case $\alpha_1<\alpha_2$.} By stationarity we have from SRE \eqref{componentSRE1} that
\[
 Y_{1,0} = D_0+ M_{11,0} Y_{1,-1}, 
\]
where $Y_{1,-1}$ has the same law as $Y_{1,0}$ and independent of
 $M_{11,0}$. {We apply} Theorem 2.3 Case 2 of
 Goldie (1991)\nocite{goldie:1991} that states that if
\begin{align*}
 I_+ =\int_0^\infty \big|
\P(Y_{1,-1}>x)-\P(M_{11,0}Y_{1,-1}>x)
\big| x^{\alpha_1-1} dx <\infty
\end{align*}
and
\begin{align*}
 I_- = \int_0^\infty \big|
\P(Y_{1,-1}<-x)-\P(M_{11,0}Y_{1,-1}<-x)
\big| x^{\alpha_1-1} dx <\infty,
\end{align*}
then
\begin{align}
& \lim_{x\to\infty} \P(Y_{1,0}>x)x^{\alpha_1} =
 \lim_{x\to\infty}\P(Y_{1,0}<-x) x^{\alpha_1} \nonumber \\ 
&\qquad =\frac{1}{2m_1} \int_0^\infty
 (\P(|Y_1|>x)-\P(|M_{11}Y_1|>x))x^{\alpha_1-1} dx. 
\label{eq:goldie:case2}
\end{align}
Due to Goldie (1991, Lemma 9.4)\nocite{goldie:1991}, if $I_+,\,I_-<\infty$ then
 the right-hand side in \eqref{eq:goldie:case2} equals $\ov c_1$, given in Theorem \ref{thm:triangular}, where we notice that 
 $|x|^{\alpha_1}=x_+^{\alpha_1}+x_{-}^{\alpha_1}$ for any $x\in \R$.
 We focus on showing that $I_+<\infty$ since the proof of $I_-<\infty$ follows by similar arguments replacing $Y_{1,-1}$ by $(-Y_{1,-1})$ in $I_+$. 
In view of Goldie (1991, Lemma 9.4)\nocite{goldie:1991}, we have that
\begin{align*}
 I_+ &= \int_0^\infty \big|
\P(Y_{1,-1}>x)-\P(M_{11,0}Y_{1,-1}>x)
\big| x^{\alpha_1-1} dx \\
&= \int_0^\infty \big|
\P(D_0+M_{11,0}Y_{1,-1}>x)-\P(M_{11,0}Y_{1,-1}>x)
\big| x^{\alpha_1-1} dx \\
&= \frac{1}{\alpha_1}\E \big[ \big|
(D_0+M_{11,0}Y_{1,-1})_+^{\alpha_1} - (M_{11,0}Y_{1,-1})_+^{\alpha_1}
\big| \big],
\end{align*}
which holds regardless of whether $I_+$ is finite or infinite.
By elementary inequalities we observe that for $\alpha_1 \le 1$,
\[
 \big|
(D_0+M_{11,0}Y_{1,-1})_+^{\alpha_1} - (M_{11,0}Y_{1,-1})_+^{\alpha_1}
\big| \le |D_0|^{\alpha_1}.
\]
Using that $\alpha_1 < \alpha_2$, we have that $E[|D_0|^{\alpha_1}]<\infty$such that $I_+<\infty$. It remains to consider the case $\alpha_1>1$, where we note that
\begin{align*}
&\big|
(D_0+M_{11,0}Y_{1,-1})_+-(M_{11,0}Y_{1,-1})_+
\big|^{\alpha_1} \\
&\quad \le
 \big|
(D_0+M_{11,0}Y_{1,-1})_+^{\alpha_1} - (M_{11,0}Y_{1,-1})_+^{\alpha_1}
\big| \\
&\quad \le \alpha_1 (D_0+M_{11,0}Y_{1,-1})_+^{\alpha_1-1} \big\{
(D_0+M_{11,0}Y_{1,-1})_+-(M_{11,0}Y_{1,-1})_+
\big\} I_{\{D_0>0\}} \\
&\qquad +\alpha_1 (M_{11,0}Y_{1,-1})_+^{\alpha_1-1} \big\{
(M_{11,0}Y_{1,-1})_+ -
(D_0+M_{11,0}Y_{1,-1})_+
\big\} I_{\{D_0<0\}} \\
& \quad \le \alpha_1  (|D_0|+|M_{11,0}Y_{1,-1}|)^{\alpha_1-1} |D_0|. 
\end{align*}
Thus we have
\[
 I_+ \le c \big(
\E|D_0|^{\alpha_1}+\E|M_{11,0}Y_{1,-1}|^{\alpha_1-1}|D_0|
\big), 
\]
where we use Minkowski's inequality and sub-additivity of concave
 functions depending on whether $\alpha_1>2$ or $1<\alpha_1 \le 2$. We
 need to prove that $\E[|M_{11,0} Y_{1,-1}|^{\alpha_1-1}|D_0|]<\infty$. Note that
\begin{align*}
& \E[|M_{11,0} Y_{1,-1}|^{\alpha_1-1}|D_0|] \\
& \le \E[|M_{11,0} Y_{1,-1}|^{\alpha_1-1}(|Q_{1,0}|+|M_{12,0}Y_{2,-1}|)] \\
& \le \E[|M_{11,0}|^{\alpha_1-1}|Q_{1,0}|] \E [|Y_{1,-1}|^{\alpha_1-1}] + 
\E[|M_{11,0}|^{\alpha_1-1} |M_{12,0}|]\, \E[|Y_{1,-1}|^{\alpha_1-1}|Y_{2,-1}|].
\end{align*}
By H\"older's inequality $\E|M_{11,0}|^{\alpha_1-1}|Q_{1,0}|$ and
 $\E|M_{11,0}|^{\alpha_1-1}|M_{12,0}|$ are finite, since all quantities
 included have finite moments of any (finite) order. We now show that $\E|Y_{1,-1}|^{\alpha_1-1}|Y_{2,-1}|$ is finite, and note that $\E |Y_{1,-1}|^{\alpha_1-1}<\infty$ follows by a similar argument.
Choose some small $\varepsilon >0$ such that $p:=(\alpha_1 - \varepsilon)/(\alpha_1-1)>0$ and $q:=p/(p-1)<\alpha_2$. By H\"older's inequality,
\[
 \E  |Y_{1,0}|^{\alpha_1-1}|Y_{2,0}|  \le \big(
\E |Y_{1,0}|^{p(\alpha_1-1)}
\big)^{1/p} \big(
\E |Y_{2,0}|^{q} 
\big)^{1/q}.
\]
With $\beta:=\alpha_1 - \varepsilon >0$ by applying Minkowski's inequality to $Y_{1,0}=\sum_{i=0}^\infty
 \Pi^{(1)}_{0,1-i}D_{-i}$, we obtain 
\[
\big(\E |Y_{1,0}|^\beta \big)^{1/\beta} \le  \sum_{i=0}^\infty (\E
 |\Pi_{0,1-i}^{(1)} D_{-i}|^\beta )^{1/\beta} =
 \sum_{i=0}^\infty \big( \E|M_{11,0}|^{\beta}
\big)^{i/\beta} \big(
\E |D_0|^\beta 
\big)^{1/\beta} <\infty,
\]
since $\E|M_{11,0}|^{\beta}<1$ by convexity and $\E |D_0|^\beta <\infty$. We conclude that $I_+<\infty$ for $\alpha_1>1$. This finishes the proof. 
\end{proof}

\appendix

\newpage

{
}


\begin{thebibliography}{99}
\baselineskip12pt
\bibitem{alsmeyer:mentmeier2012}
{\sc Alsmeyer, G. and Mentemeier, S.}\ (2012)
Tail behavior of
stationary solutions of random difference equations: the case of
regular matrices.
 {\em Journal of Differential Equations and Applications}, {\bf 18}, 1305--1332.

\bibitem{avarucci:beutner:zaffaroni2013}
{\sc Avarucci, M. and Beutner, E. and Zaffaroni, P.}\ (2013)
On moment conditions for quasi-maximum likelihood estimation of multivariate ARCH models.
 {\em Econometric Theory}, {\bf 29}, 545--566.

\bibitem{bougerol:picard:1992}
{\sc Bougerol, P and Picard N.}\ (1992)
Strict stationarity of generalized autoregressive processes.
{\em Annals of Probability}, {\bf 20}, 1714--1730.

\bibitem{brandt:1986}
{\sc  Brandt, A.} (1986) The stochastic equation
$Y_{n+1}=A_n\,Y_n+B_n$ with stationary coefficients.
{\em Advances in Applied Probability},  {\bf 18}, 211--220.

\bibitem{breiman:1965}
{\sc  Breiman, L.} (1965) On some limit theorems similar to the arc-sin law.
{\em Theory of Probability and its Applications},  {\bf 10}, 323--331.

\bibitem{boussama:fuchs:stelzer:2011}
{\sc  Boussama, F., Fuchs, F., and Stelzer, R.} (2011) Stationarity and geometric ergodicity of BEKK multivariate GARCH models.
{\em Stochastic Processes and their Applications},  {\bf 121}, 2331--2360.


\bibitem{buraczewski:damek:guivarch2009}
{\sc Buraczewski, D., Damek, E., Guivarc'h, Y., Hulanicki, A. and Urban,
	R.}\ (2009)
Tail-homogeneity of stationary measures for some multidimensional
  stochastic recursions.
{\em Probability Theory and Related Fields}, {\bf 145}, 385--420.


\bibitem{buraczewski:damek:mikosch:2016}
{\sc Buraczewski, D., Damek, E. and Mikosch, T.}\ (2016)
{\em Stochastic Models with Power-Law Tails: The Equation $X=AX+B$}.
Springer Series in Operations Research and Financial Engineering,
	Springer International Publishing. 

\bibitem{cont:2001}
{\sc Cont, R.}\ (2001)
Empirical properties of asset returns: stylized facts and statistical issues.
{\em Quantitative Finance}, {\bf 1}, 223-236.

\bibitem{damek:matsui:swiatkowski:2017}
{\sc Damek, E., Matsui, M. and \'{S}wi\k{a}tkowski, W.}\ (2019)
Componentwise different tail solutions for bivariate stochastic
	recurrence equations with application to GARCH(1,1)
processes. {\em Colloquium Mathematicum}, {\bf 155}, 227--254.

\bibitem{damek:zienkiewicz:2017}
{\sc Damek, E. {and} Zienkiewicz, J.}\ (2018)
Affine stochastic equation with triangular matrices.
{\em Journal of Differential Equations and Applications}, {\bf 24}, 520--542.

\bibitem{davis:mikosch:2009}
{\sc Davis, R.A. {and} Mikosch, T.}\ (2009)
Extreme value theory for GARCH processes. In Andersen, T.G., Davis, R.A., Krei{\ss}, J.-P. and Mikosch T. (Eds.),
{\em Handbook of Financial Time Series}, 186-200. Springer, New York, NY.

\bibitem{engle:kroner:1995}
{\sc Engle, R.F and Kroner K.F.}\ (1995)
Multivariate simultaneous generalized ARCH.
{\em Econometric Theory}, {\bf 11}, 122--150.

\bibitem{francq:zakoian:2010}
{\sc Francq, C. and Zako{\"i}an}\ (2010)
{\em GARCH Models: Structure, Statistical Inference and Financial Applications}.
John Wiley \& Sons, Chichester.


\bibitem{gabaix:2009}
{\sc Gabaix, X.}\ (2009)
Power laws in economics and finance.
{\em The Annual Review of Economics}, {\bf 1}, 255--293.

\bibitem{GGO}
{\sc Gerencs\'er, L., Michaletzky, G. and Orlovits, Z.}\ (2008)
Stability of block-triangular stationary random matrices.
{\em Systems \& Control Letters}, {\bf 57}, 620--625.

\bibitem{goldie:1991}
{\sc Goldie, C.M.}\ (1991)
Implicit renewal theory and tails of solutions of random equations.
{\em Annals of Applied Probability}, {\bf 1}, 126--166.

\bibitem{guivarch:lepage:2016}
{\sc Guivarc'h, Y. and Le Page, \'E.}\ (2016)
Spectral gap properties for linear random walks and Pareto’s asymptotics for affine stochastic recursions.
{\em Annales de l'Institut Henri Poincar\'e (B) Probabilit\'es et Statistiques}, {\bf 52}, 503--574.


\bibitem{horn:johnson:2013}
{\sc Horn, R.A. and Johnson, C.R.}\ (2013)
{\em Matrix Analysis}, 2nd edition, Cambridge University Press.

\bibitem{ibragimov:ibragimov:walden:2015}
{\sc Ibragimov, M., Ibragimov, R. and Walden, J.}\ (2015)
{\em Heavy-Tailed Distributions and Robustness in Economics and Finance. Lecture Notes in Statistics} 
{\bf 214}.
Springer International Publishing, Switzerland.

\bibitem{krengel:2011}
{\sc Krengel, U.} (2011)
{\em Ergodic theorems}. Vol. 6. 
Walter de Gruyter, {Berlin}.

\bibitem{ling:li:2008}
{\sc Ling, S. and Li, D.}\ (2008)
Asymptotic inference for a nonstationary double AR(1) model.
{\em Biometrika}, {\bf 95}, 257--263.


\bibitem{loretan:phillips:1994}
{\sc Loretan, M. and Phillips, P.C.B.}\ (1994)
Testing for covariance stationarity of heavy-tailed time series.
{\em Journal of Empirical Finance}, {\bf 1}, 211--248.

\bibitem{matsui:mikosch:2016}
{\sc Matsui, M. and Mikosch, T.}\ (2016)
The extremogram and the cross-extremogram for a bivariate
        GARCH(1,1) process.
{\em Advances in Applied Probability}, {\bf 48A}, 217--233.

\bibitem{matsui:swiatkowski:2018}
{\sc Matsui, M. and \'{S}wi\k{a}tkowski, W.}\ (2018)
Tail indices for AX+B recursion with triangular matrices.
{\em arXiv:1808.09678}.

\bibitem{nelson:1990}
{\sc Nelson, D.B.}\ (1990)
Stationarity and persistence in the GARCH(1,1) model.
{\em Econometric Theory}, {\bf 6}, 318--334.

\bibitem{nicholls:quinn:1982}
{\sc Nicholls, D.F. and Quinn, B.G.}\ (1982)
{\em Random Coefficient Autoregressive Models: An Introduction. Lecture Notes in Statistics} 
{\bf 11}.
Springer, New York, NY.

\bibitem{nielsen:rahbek:2014}
{\sc Nielsen, H.B. and Rahbek, A.}\ (2014)
Unit root vector autoregression with volatility induced stationarity.
{\em Journal of Empirical Finance}, {\bf 29}, 144--167.


\bibitem{pedersen:2016}
{\sc Pedersen R.S.}\ (2016)
Targeting estimation of CCC-GARCH models with infinite fourth moments.
{\em Econometric Theory}, {\bf 32}, 498--531.

\bibitem{pedersen:rahbek:2014}
{\sc Pedersen R.S. and Rahbek, A.}\ (2014)
Multivariate variance targeting in the BEKK-GARCH model.
{\em The Econometrics Journal}, {\bf 17}, 24--55.

\bibitem{pedersen:wintenberger:2018}
{\sc Pedersen R.S. and Wintenberger, O.}\ (2018)
On the tail behavior of a class of multivariate conditionally
	heteroskedastic processes.
{\em Extremes}, {\bf 21}, 261--284.

\bibitem{resnick:2007}
{\sc Resnick, S.I.}\ (2007)
{\em Heavy-Tail Phenomena: Probabilistic and Statistical Modeling}.
Springer, New York, NY.

\bibitem{starica:1999}
{\sc St{\u{a}}ric{\u{a}}, C.}\ (1999)
Multivariate extremes for models with constant conditional correlations.
{\em Journal of Empirical Finance}, {\bf 6}, 515--553.

\bibitem{S}
{\sc Straumann, D.}\ (2005)
{\em Estimation in conditionally heteroscedastic time series models.}
Lecture Notes in Statistics {\bf 181}, Springer-Verlag, Berlin.

\bibitem{sun:zhou:2014}
{\sc Sun, P. and Zhou C.}\ (2014)
Diagnosing the distribution of GARCH innovations.
{\em Journal of Empirical Finance}, {\bf 29}, 287--303.


\end{thebibliography}
\end{document}